\documentclass[smallextended]{svjour3}       
\smartqed  
\usepackage{graphicx,amsmath,amssymb,tikz}
\usepackage{pgf, tikz,url}
\usetikzlibrary{arrows, automata}
\usepackage{times,amsmath,graphicx,amssymb}
\usepackage{tabstackengine,bm}
\stackMath

\RequirePackage{fix-cm}
\newcommand{\reals}{\mathbb{R}}
\newcommand{\vz}{\bm{z}}
\newcommand{\vy}{\bm{y}}
\newcommand{\vT}{\bm{T}}
\newcommand{\vU}{\bm{U}}

\newcommand{\zer}{\mathop{\bf zer}}

\makeatletter

\renewcommand{\paragraph}{%
  \@startsection{paragraph}{4} {\z@}
  {1ex \@plus 1ex \@minus .2ex}
  {-1ex}%
  {\bfseries}%
}                                  
 
\makeatother

\begin{document}
\title{Uniqueness of DRS as the 2 Operator Resolvent-Splitting\\
and
 Impossibility of 3 Operator Resolvent-Splitting}
\titlerunning{Uniqueness and Impossibility of 2 and 3 Operator Resolvent-Splitting}
\author{Ernest K. Ryu}
\institute{Ernest K. Ryu\at
              7324 Mathematical Sciences,\\
              UCLA\\
              Los Angeles, CA 90095 \\
              \email{eryu@math.ucla.edu}
}

\date{Received: 21 February 2018 / Accepted: 03 April 2019}
\maketitle

\begin{abstract}
Given the success of Douglas--Rachford splitting (DRS), it is natural to ask whether DRS can be generalized. Are there other 2 operator resolvent-splittings sharing the favorable properties of DRS? Can DRS be generalized to 3 operators? This work presents the answers: no and no. In a certain sense, DRS is the unique 2 operator resolvent-splitting, and generalizing DRS to 3 operators is impossible without lifting, where lifting roughly corresponds to enlarging the problem size. The impossibility result further raises a question. How much lifting is necessary to generalize DRS to 3 operators? This work presents the answer by providing a novel 3 operator resolvent-splitting with provably minimal lifting that directly generalizes DRS.

\noindent
{\color{red}(Update: A few minor typos have been corrected. The changes are highlighted in red.)}

\vspace{0.2in}

\end{abstract}

\section{Introduction}
In 1979, Lions and Mercier 
presented Douglas--Rachford splitting (DRS)
which solves the monotone inclusion problem 
\[
\underset{x\in \mathbb{R}^d}{\mbox{find}}\quad
0\in (A+B)x
\]
with
\[
z^{k+1}=(1-\theta/2)z^k+(\theta/2)(2J_{\alpha A}-I)(2J_{\alpha B}-I)z^k
\]
for any $\alpha>0$ and $\theta\in (0,2)$,
where $A$ and $B$ are maximal monotone operators
and $J_{\alpha A}$ and $J_{\alpha B}$ are their resolvents
 \cite{peaceman1955,douglas1956,lions1979}.
 Since its introduction, DRS has enjoyed great popularity 
and has provided great value to the field of optimization.

Given the success of DRS, one may ask the following two questions:
\begin{enumerate}
\item Are there other 2 operator resolvent-splittings?
\item Can we generalize DRS to 3 operators?
\end{enumerate}
In fact, the second question has been a long-standing open problem posed by Lions and Mercier themselves:
``[T]he convergence seems difficult to prove ... in the case of a sum of 3 operators.''
After all, identifying why a tool works and generalizing it
is a common and often fruitful exercise in mathematics.

This work presents the answers to these questions: no and no.
In a certain sense, DRS is the unique 2 operator resolvent-splitting.
In a certain sense, there is no 3 operator resolvent-splitting without lifting,
where lifting roughly corresponds to enlarging the problem size.

This impossibility result further raises the following question:
\begin{enumerate}
\item[3.] To generalize DRS to 3 operators, how much lifting is necessary?
\end{enumerate}
This work presents the answer
by providing a novel 3 operator resolvent-splitting with provably minimal lifting.

\paragraph{Background.}
To discuss what constitutes a generalization of DRS,
we first point out a few key properties of DRS.
Perhaps a generalization of DRS should satisfy these as well.
\begin{enumerate}
\item
DRS is a \emph{resolvent-splitting} in that it is constructed with scalar multiplication, addition, and resolvents.
\item
DRS is \emph{frugal} in that it uses $J_{\alpha A}$ and $J_{\alpha B}$ only once per iteration.
\item
DRS \emph{converges unconditionally} in that 
it works for any maximal monotone $A$ and $B$.
\item
DRS uses \emph{no lifting} in that
the fixed-point mapping maps from
$\mathbb{R}^d$ to $\mathbb{R}^d$,
where $x\in \mathbb{R}^d$.
In other words, DRS does not enlarge the problem size.
\end{enumerate}

Consider the  
 proximal point method (PPM) \cite{martinet1970,martinet1972,rockafellar1976,brezis1978},
which finds an $x\in \reals^d$ such that $0\in Ax$
with
\[
x^{k+1}=J_{\alpha A}x^k
\]
for any $\alpha>0$ and maximal monotone $A$.
DRS generalizes PPM, and both methods are frugal, converge unconditionally, use no lifting,
and rely on resolvents.
Therefore, to require the 4 properties in a generalization of DRS seems reasonable.


Many other splittings have been presented since DRS,
and they have certainly provided great value to the field of optimization.
These splittings solve a wide range of  different problem classes
and are designed to be effective under a wide range of different computational considerations.
Many of them include DRS as a special case and therefore are generalizations of DRS, in that sense.
However, they do not satisfy the 4 stated properties and
therefore are not generalizations of DRS, in this sense.

Forward-backward splitting (FBS) \cite{passty1979},
\[
x^{k+1}=J_{\alpha B}(I-\alpha A)x^k,
\]
which requires $A$ to be cocoercive,
is frugal, uses no lifting, but is not a resolvent-splitting.
Primal-dual hybrid gradient method (PDHG) \cite{zhu08,pock2009,esser2010,ChaPoc11}, also known as Chambolle--Pock, 
\begin{align*}
x^{k+1}&=J_{A}(x^k-\alpha u^k)\\
u^{k+1}&=(I-J_{B})(u^k+\alpha (2x^{k+1}-x^k))
\end{align*}
is frugal but uses lifting.
Davis--Yin splitting (DYS) \cite{davis2017}, which finds an $x\in \reals^d$ such that $0\in (A+B+C)x$,
where $C$ is cocoercive,
\begin{align*}
z^{k+1}=
(I-J_{\alpha B}+J_{\alpha A}\circ(2J_{\alpha B}-I-\alpha C\circ J_{\alpha B}))z^k
\end{align*}
is frugal, uses no lifting, but is not a resolvent-splitting.
Other methods, such as 
FBFS \cite{tseng2000},
PPXA \cite{combettes2008,combettes2011},
PDFP\textsuperscript{2}O/PAPC \cite{loris2011,chen2013,driori2015},
RFBS \cite{banert2012},
Condat--V\~u \cite{condat2013,vu2013},
GFBS \cite{raguet2013},
PD3O \cite{yan2016},
PDFP \cite{Chen2016},
AFBA \cite{latafat2017},
FBHFS \cite{ba2017},
FDRS \cite{BAFDRS,Raguet2018},
FRB \cite{malitsky2018},
projective splitting \cite{Eckstein2017,Combettes2018,johnstone2018,johnstone2019},
and the methods of \cite{ba2011,Combettes2012,bot_primal-dual_2013,combettes2014}
all fail to satisfy the 4 properties.


\paragraph{Organization of the paper.}
In Section~\ref{s:uniqueness},
we show that 
DRS is the only frugal, unconditionally convergent resolvent-splitting without lifting 
for the 2 operator problem.
We do so by characterizing all frugal resolvent-splittings without lifting
and showing that DRS is the only one among them
that unconditionally converges.

In Section~\ref{s:impossibility},
we show that there is no resolvent-splitting without lifting for the 3 operator problem,
even if the splitting is not frugal and not convergent.
In particular, we show such a scheme without lifting
cannot be a fixed-point encoding.

In Section~\ref{s:attainment}, we define and quantify the notion of  lifting for the 3 operator problem.
We then provide a novel 
frugal, unconditionally convergent resolvent-splitting with provably minimal lifting
for the 3 operator problem that directly generalizes DRS.

\paragraph{Definitions.}
We briefly review some standard notation and results of operator theory.
Interested readers can find in-depth discussion of these concepts
in standard references such as \cite{ryu2016,BauschkeCombettes2017_convex}.

Write $\langle\cdot,\cdot\rangle$ for the standard Euclidean inner product in $\reals^d$.
We say $A$ is an operator on $\reals^d$
if $A$ maps points of $\reals^d$ to subsets of $\reals^d$.
Given a matrix $M\in \reals^{d\times d}$ also  write $M:\reals^d\rightarrow\reals^d$
to denote the linear operator defined by the matrix $M$.
In particular, write $I$ for both the identity operator and the identity matrix.
Write $\mathcal{M}(\reals^d)$ for the set of all maximal monotone operators on $\reals^d$.
For any maximal monotone operator $A$ and $\alpha>0$, write
\[
J_{\alpha A} =(I+\alpha A)^{-1}
\]
for the resolvent of $A$.
A mapping $T:\reals^d\rightarrow\reals^d$ is nonexpansive if
\[
\|Tx-Ty\|^2\le \|x-y\|^2
\]
for all $x,y\in \reals^d$.
A mapping $F:\reals^d\rightarrow\reals^d$ is firmly nonexpansive if
\[
\|Fx-Fy\|^2 \le \langle x-y,Fx-Fy\rangle
\]
for all $x,y\in \reals^d$.
Resolvents are firmly nonexpansive.
Given a mapping $T:\reals^d\rightarrow\reals^d$ and a starting point $z^0\in \reals^d$, we call
\[
z^{k+1}=Tz^k
\]
the fixed-point iteration with respect to $T$.
A fixed-point iteration with respect to a nonexpansive mapping
need not converge.
A mapping $T:\reals^d\rightarrow\reals^d$ is averaged if 
it can be expressed as $T=(1-\theta)I+\theta R$, where $R:\reals^d\rightarrow\reals^d$ is nonexpansive and $\theta\in (0,1)$.
Note that $R$ and $T$ share the same fixed points.
The fixed-point iteration with respect to an averaged mapping
$T:\reals^d\rightarrow\reals^d$
converges in that $z^k\rightarrow z^\star$ where $Tz^\star=z^\star$, if a fixed point exists.

For any $A\in \mathcal{M}(\reals^d)$,
write $\zer A=\{x\,|\,0\in Ax\}$ for the set of zeros of $A$.
Consider the monotone inclusion problem of finding
an element of $\zer(A+B)$, where $A,B\in \mathcal{M}(\reals^d)$.
Peaceman--Rachford splitting (PRS) \cite{peaceman1955,lions1979}
is the fixed-point iteration
\[
z^{k+1}=(2J_{\alpha A}-I)(2J_{\alpha B}-I)z^k
\]
with $\alpha>0$.
PRS is not guaranteed to converge.
Douglas--Rachford splitting (DRS)
is the fixed-point iteration
\[
z^{k+1}=(1-\theta/2)z^k+(\theta/2) (2J_{\alpha A}-I)(2J_{\alpha B}-I)z^k.
\]
with $\alpha>0$ and $\theta\in (0,2)$.
(Some may call this ``relaxed PRS''.)
DRS is guaranteed to converge
in the sense that $z^k\rightarrow z^\star$ for some $z^\star$
where $J_{\alpha B}z^\star\in \zer(A+B)$, if $\zer(A+B)$ is not empty.

\section{Uniqueness of DRS as the unique
frugal, unconditionally convergent 2 operator resolvent-splitting without lifting}
\label{s:uniqueness}
In this section, we define what a
frugal, unconditionally convergent 2 operator resolvent-splitting without lifting is
and prove  DRS is the only such splitting.

\subsection{Definitions}
When reading the definitions, it is helpful to think of DRS as a specific example.
In the terminology and notation we soon establish, 
DRS is an unconditionally convergent frugal resolvent-splitting without lifting
and 
$d'=d$,
$T(A,B,z)=(1-\theta/2)I+(\theta/2)(2J_{\alpha A}-I)(2J_{\alpha B}-I)$, and
$S(A,B,z)=J_{\alpha B}$.

Given a dimension $d$,
define the problem class \eqref{eq:2op} to be the collection of monotone inclusion problems of the form
\begin{equation}
\underset{x\in \mathbb{R}^d}{\mbox{find}}\quad
0\in (A+B)x
\tag{2op-$\reals^d$}
\label{eq:2op}
\end{equation}
with $A, B\in \mathcal{M}(\reals^d)$.

\paragraph{Fixed-point encoding.}
A pair of functions $(T,S)$ is a \emph{fixed-point encoding}
for the problem class \eqref{eq:2op}
if
\[
\exists z^\star\in \reals^{d'} \text{ such that}\,
\left(
\begin{array}{ll}
T(A, B, z^\star)&=z^\star\\
S(A, B, z^\star)&=x^\star
\end{array}
\right)
\quad\Leftrightarrow\quad
0\in (A+B)(x^\star)
\]
for all $A,B\in \mathcal{M}(\reals^d)$.
We call
\[
T:\mathcal{M}(\reals^d)\times \mathcal{M}(\reals^d)\times \reals^{d'}\rightarrow \reals^{d'}
\]
the \emph{fixed-point mapping}
and
\[
S:\mathcal{M}(\reals^d)\times \mathcal{M}(\reals^d)\times \reals^{d'}\rightarrow \reals^{d},
\]
the \emph{solution mapping}.
To clarify, a fixed-point encoding is defined for the entire problem class \eqref{eq:2op},
rather than a single instance of the monotone inclusion problem.

When we fix $A,B\in \mathcal{M}(\reals^d)$,
fixed points of $T(A,B,\cdot):\reals^{d'}\rightarrow\reals^{d'}$
corresponds to zeros of $A+B$.
We say that  points in $\zer(A+B)$
 are \emph{encoded}
as fixed points of  $T(A,B,\cdot)$.
For notational simplicity, we often drop the dependency on $A$ and $B$ and write $Tz$ and $Sz$
for $T(A,B,z)$ and $S(A,B,z)$.

In this section, we only consider $d'=d$, as we limit our attention to fixed-point encodings without lifting (formally defined soon).
In general, however, the dimension $d$ of problems in \eqref{eq:2op}
and the dimension $d'$ of the fixed-point mapping need not be the same.
The purpose of allowing $d'\ne d$ will become clearer
later in Section~\ref{s:attainment},
where an analogously defined $d'$ is larger than $d$.

Under this definition, DRS is a \emph{collection} of fixed-point encodings.
For each choice of $d$, $\alpha>0$, $\theta\in(0,2)$, and $\eta\in \reals$, the pair of functions $(T,S)$ defined by
\begin{align*}
T(A,B,z)&=(1-\theta/2)z+(\theta/2) (2J_{\alpha {\color{red}B}}-I)(2J_{\alpha {\color{red}A}}-I)z,\\
S(A,B,z)&=\eta J_{\alpha A}z+(1-\eta)J_{\alpha B}(2J_{\alpha A}-I)z
\end{align*}
is a instance of DRS and it is a fixed-point encoding for the problem class \eqref{eq:2op}.

\paragraph{Frugal resolvent-splitting without lifting.}
Loosely speaking, $(T,S)$ is a \emph{resolvent-splitting} for the problem class \eqref{eq:2op} if it is a fixed-point encoding 
constructed with resolvents of $A$ and $B$,
addition, and scalar multiplication.
Loosely speaking, 
$(T,S)$ is \emph{frugal} if it uses $J_{\alpha A}$ and $J_{\beta B}$ once,
in that a single evaluation of $J_{\alpha A}$ and a single evaluation of $J_{\beta B}$
is used to evaluate \emph{both} $Tz$ and $Sz$ for any $z$.
$(T,S)$ is \emph{without lifting} if
$T(A,B,\cdot):\reals^d\rightarrow\reals^d$
and
$S(A,B,\cdot):\reals^d\rightarrow\reals^d$ for any $A,B\in \mathcal{M}(\reals^d)$,
i.e., if $d'=d$.

We now make the definitions precise.
Let $I$ be the ``identity mapping'' defined as
$I:\mathcal{M}(\reals^d)\times\mathcal{M}(\reals^d)\times\reals^d\rightarrow\reals^d$ 
and
$I(A,B,z)=z$ for any $A,B\in \mathcal{M}(\reals^d)$ and $z\in \reals^d$.
Let $J_{\alpha,1}$ be the resolvent with respect to the first operator
defined as
$J_{\alpha ,1}:\mathcal{M}(\reals^d)\times\mathcal{M}(\reals^d)\times\reals^d\rightarrow\reals^d$ 
and
$J_{\alpha ,1}(A,B,z)=J_{\alpha A}(z) $ for any $A,B\in \mathcal{M}(\reals^d)$ and $z\in \reals^d$.
Define  $J_{\beta,2}$ likewise with $J_{\beta ,2}(A,B,z)=J_{\beta B}(z) $.
Define the class of mappings 
\[
\mathcal{F}_0=\{I\}\cup
\{J_{\alpha,1}\,|\,\alpha>0\}\cup
\{J_{\beta,2}\,|\,\beta>0\}.
\]
Recursively define
\[
\mathcal{F}_{i+1}=
\{F+G\,|\,F,G\in \mathcal{F}_{i}\}
\cup
\{F\circ G\,|\,F,G\in \mathcal{F}_{i}\}
\cup
\{\gamma F\,|\,F\in \mathcal{F}_i,\,\gamma\in \reals\}
\]
for $i=0,1,2,\dots$.
The ``composition'' $F\circ G$ is defined with
\[
(F\circ G)(A,B,z)=F(A,B,G(A,B,z))
\]
 for any $z\in \reals^d$ and $A,B\in \mathcal{M}(\reals^d)$.
Note that $\mathcal{F}_0\subset \mathcal{F}_1\subset \mathcal{F}_2\subset \cdots$. Finally define
\[
\mathcal{F}=\bigcup^\infty_{i=0}\mathcal{F}_i.
\]
To clarify, 
elements of $\mathcal{F}$ map
$\mathcal{M}(\reals^d)\times\mathcal{M}(\reals^d)\times\reals^d$
to $\reals^d$.
If $R\in \mathcal{F}$ and $A,B\in \mathcal{M}(\reals^d)$,
then $R(A,B,\cdot):\reals^d\rightarrow\reals^d$.
These mappings are constructed with (finitely many) resolvents of $A$ and $B$, addition, and scalar multiplication.

As an aside, 
we could have defined $\mathcal{F}$ as the ``near-ring''
generated by $J_{\alpha,1}$ and $J_{\beta ,2} $ for all $\alpha>0$ and $\beta>0$
and $\gamma I$ for all $\gamma\in \reals$. 
The set is not a ring because $T\circ(U+V)\ne T\circ U+T\circ V$ for non-linear functions.

We say $(T,S)$ is a \emph{resolvent-splitting without lifting} for the problem class \eqref{eq:2op},
if $(T,S)$ is a fixed-point encoding for the problem class \eqref{eq:2op}, and $T,S\in \mathcal{F}$.
(Remember, $T\in \mathcal{F}$ implies $T(A,B,\cdot):\reals^d\rightarrow\reals^d$.)


When $T,S\in \mathcal{F}$, one can evaluate $T(A,B,z)$ and $S(A,B,z)$ for given $z\in \reals^{d}$ and $A,B\in \mathcal{M}(\reals^d)$
in finitely many steps, where each step is scalar multiplication, vector addition, or a resolvent evaluation.
We say $(T,S)$ is \emph{frugal} if it has a step-by-step (serial) evaluation procedure
such that exactly one step computes $J_{\alpha A}$, exactly one step computes $J_{\beta B}$,
and both $T(A,B,z)$ and $S(A,B,z)$ are output at the end.

\paragraph{Unconditional convergence.}
We say $(T,S)$ \emph{converges unconditionally} for the problem class \textup{(2op-$\reals^d$)}
if
\[
T^kz^0 \rightarrow z^\star,\quad Sz^\star\in \zer(A+B)
\]
for any $z^0\in \reals^{d'}$ and $A,B\in \mathcal{M}(\reals^d)$
as $k\rightarrow\infty$, when $\zer(A+B)\ne\emptyset$.
To clarify, 
\[
T^k=\underbrace{T\circ T\circ \cdots \circ T}_{\text{$k$ times}}.
\]
We say the convergence is unconditional because there are no conditions on the operators $A,B\in \mathcal{M}(\reals^d)$
or the starting point $z^0\in \reals^d$.

For example, with DRS, the $z^k$-iterates do not, in general, converge to a solution.
Rather, $z^k\rightarrow z^\star$, where $J_{\alpha B}z^\star$ is a solution to the monotone inclusion problem, when a solution exists.

The notion of unconditional convergence is unrelated to weak and strong convergence.
In infinite dimensional spaces, we would require the convergence $T^kz^0 \rightarrow z^\star$ to hold weakly,
but weak and strong convergence coincide in finite dimensions.
We avoid infinite dimensional spaces because defining the notion of lifting would be awkward when $d=\infty$.


\paragraph{Equivalence.}
Given a fixed-point iteration, we can scale it with a nonzero scalar to get another one
that is essentially the same, i.e.,
\[
z^{k+1}=T(z^k)
\quad\Leftrightarrow\quad
az^{k+1}=aT(a^{-1} az^k)
\]
for any $a\in \reals$ such that $a \ne 0$.
Given resolvent-splitting, we can swap the role of $A$ and $B$
to get another one that is conceptually no different, i.e.,
\[
(T(A,B,\cdot),S(A,B,\cdot))
\quad\Leftrightarrow\quad
(T(B,A,\cdot),S(B,A,\cdot)).
\]
Two resolvent-splittings without lifting are \emph{equivalent}
if one can be obtained from the other through scaling with a nonzero scalar and/or swapping the role of $A$ and $B$.

\subsection{Uniqueness result}

\begin{theorem}
\label{thm:2op-necessary}
Up to equivalence, 
$(T,S)$ is a frugal resolvent-splittings without lifting for the problem class \eqref{eq:2op} 
if and only if it is of the form
\begin{align*}
x_1&=J_{\alpha A}z\\
x_2&=J_{\beta B}((1+\beta/\alpha)x_1-(\beta/\alpha) z)\\
T(z)&=z+\theta(x_2-x_1)\\
S(z)&=\eta x_1+(1-\eta)x_2
\end{align*}
for some $\alpha,\beta>0$, $\theta\ne 0$, and  $\eta\in \reals$.
\end{theorem}
Note that Theorem~\ref{thm:2op-necessary} says nothing about convergence.
Theorem~\ref{thm:2op-sufficiency} characterizes the splittings of Theorem~\ref{thm:2op-necessary} that do converge.


\begin{theorem}
\label{thm:2op-sufficiency}
$(T,S)$ of Theorem~\ref{thm:2op-necessary}
converges unconditionally if and only if $\alpha=\beta$ and $\theta\in (0,2)$ if $d\ge 2$.
\end{theorem}
When $d=1$, the splittings $(T,S)$ of Theorem~\ref{thm:2op-necessary} may converge under more general conditions, but we do not pursue this discussion.

\begin{corollary}
Up to equivalence, the class of DRS splittings (the collection parameterized by $\alpha>0$, $\theta\in(0,2)$, and $\eta\in \reals$)
are the only frugal, unconditionally convergent resolvent-splittings without lifting for the problem class \eqref{eq:2op} when $d\ge 2$.
\end{corollary}
\begin{proof}
A frugal resolvent-splittings without lifting must be equivalent to a splitting of the form of Theorem~\ref{thm:2op-necessary}.
If it is furthermore unconditionally convergent, then, by  Theorem~\ref{thm:2op-sufficiency}, we have $\alpha=\beta$ and $\theta\in(0,2)$, which corresponds to DRS.
\qed
\end{proof}


\subsection{Proof of Theorem~\ref{thm:2op-necessary}}
\subsubsection{Outline}
The main part of proof, which shows that any frugal resolvent-splitting without lifting is of the form of Theorem~\ref{thm:2op-necessary}, can be divided into in roughly three steps.
In the first step, we represent a given resolvent-splitting $(T,S)$ with a linear system of equations, and simplify the system using Gaussian elimination.
In the second step, we show that the system of linear equalities must imply certain equalities one would expect from a fixed-point encoding.
This is done by using a Farkas-type lemma 
to take a certain element from the null space of the linear system and using it to construct the counter example.
In the third step, we use the conclusion of the second step to eliminate and characterize the parameters of $(T,S)$.

\subsubsection{Proof}

Showing that $(T,S)$ of Theorem~\ref{thm:2op-necessary} is indeed a fixed-point encoding is straightforward.
Let $x^\star\in\reals^d$ satisfy $0\in (A+B)x^\star$.
Let $\tilde{A}x^\star\in Ax^\star$ and 
$\tilde{B}x^\star\in Bx^\star$ such that $\tilde{A}x^\star+\tilde{B}x^\star=0$,
and let $z_0 =x^\star+\alpha \tilde{A}x^\star$.
Then $x_1=x_2=x^\star$, $Tz_0=z_0$, and $Sz_0=x^\star$.
On the other hand, assume $T(A,B,z^\star)=z^\star$.
Then $x_1=x_2$. Write $x^\star=x_1=x_2$,
$\tilde{A}x^\star=(1/\alpha)(z^\star-x^\star)$,
and
$\tilde{B}x^\star=(1/\alpha)(x^\star-z^\star)$.
Then
$\tilde{A}x^\star\in Ax^\star$, 
$\tilde{B}x^\star\in Bx^\star$, and
$\tilde{A}x^\star+\tilde{B}x^\star=0$,
which implies $x^\star=S(z^\star)$ is a solution.

We now need to show the other direction, that any frugal resolvent-splitting without lifting for the problem class \eqref{eq:2op}
is of the form of Theorem~\ref{thm:2op-necessary}, up to equivalence.

First, we discuss the following Farkas-type lemma.
\begin{lemma}
\label{lm-ge}
Let $M\in \reals^{m\times n}$ and $c\in \reals^{n}$ be fixed coefficients,
and let $v\in \reals^n$ be a variable.
If there is a $w\in \reals^{ m}$ such that $w^TM=c^T$ then the linear equalities $Mv=0$ imply the linear equality $c^Tv=0$.
If there is no such $w$, then there is an instance of $v\in \reals^n$ such that $Mv=0$ but $c^Tv\ne 0$.
\end{lemma}
An equivalent way to state Lemma~\ref{lm-ge} is to say that $Mv=0$ implies $c^Tv=0$ if and only if we can linearly combine the rows of $Mv=0$ to obtain $c^Tv=0$.
We say Lemma~\ref{lm-ge} is of Farkas-type as it resembles Farkas' result on systems of linear inequalities \cite{Farkas1902}. For a systematic study on Farkas-type theorems, see \cite{bot_farkas_2005,dinh_farkas_2007}.
Lemma~\ref{lm-ge} can be directly and easily proved with standard linear algebra.




We now proceed onto the main proof.
Let $(T,S)$ be a frugal resolvent-splitting without lifting.


Consider an evaluation procedure of $(T,S)$ that establishes frugality.
In the step-by-step computation, either $J_{\alpha A}$ or $J_{\beta B}$ is evaluated before the other.
Without loss of generality, assume $J_{\alpha A}$ is evaluated before $J_{\beta B}$
in this ordering, since we can otherwise consider  $(\tilde{T},\tilde{S})$ defined with
\[
\tilde{T}(A,B,z)=T(B,A,z),\qquad \tilde{S}(A,B,Z)=S(B,A,z),
\]
the equivalent splitting with the order of $A$ and $B$ swapped.

Consider the evaluation of $Tz_0$ and $Sz_0$ for $z_0\in \reals^d$.
Write $z_1$ and $z_2$ for the inputs and $x_1$ and $x_2$ for the outputs of the resolvent evaluations with respect to $A$ and $B$, i.e., $x_1=J_{\alpha A}z_1$ and  $x_2=J_{\beta B}z_2$.
Define $\tilde{A}x_1$ and $\tilde{B}x_2$ with
$x_1+\alpha \tilde{A}x_1=z_1$ and $x_2+\beta \tilde{B}x_2=z_2$.
By definition of resolvents, we have $\tilde{A}x_1\in Ax_1$
and $\tilde{B}x_2\in Bx_2$.

All computational steps except the evaluations of $J_{\alpha A}$ and $J_{\beta B}$ amount to forming linear combinations of 
previous information since scalar multiplication and vector addition are the only other operations allowed in $(T,S)$.
Therefore,  we can express the evaluation of $Tz_0$ and $Sz_0$ as
\begin{equation}
\setlength\arraycolsep{5pt}
0=
\begin{bmatrix}
*&1&0&0&0&0&0&0&0\\
0&-1&1&0&0&0&\alpha&0&0\\
*&*&*&1&0&0&0&0&0\\
0&0&0&-1&1&0&0&\beta&0\\
*&*&*&*&*&1&*&*&0\\
*&*&*&*&*&*&*&*&1
\end{bmatrix}
\begin{bmatrix}
z_0\\
z_1\\
x_1\\
z_2\\
x_2\\
Tz_0\\
\tilde{A}x_1\\
\tilde{B}x_2\\
Sz_0
\end{bmatrix}.
\label{eq:eq1}
\end{equation}
Each scalar in the matrix represents a $d\times d$ block.
The symbol $*$ denotes a fixed scalar coefficient that we have not yet parameterized.
Row 1 defines $z_1$, the input to $J_{\alpha A}$.
Row 2 represents $x_1=J_{\alpha A}z_1\,\Leftrightarrow \, x_1+\alpha Ax_1\ni z_1$.
Row 3 defines $z_2$, the input to $J_{\beta B}$.
Row 4 represents $x_2=J_{\beta B}z_2\,\Leftrightarrow \, x_2+\beta Bx_2\ni z_2$.
Row 5 defines $Tz_0$.
Row 6 defines $Sz_0$.
We will simplify the system first and then explicitly parameterize the coefficients to keep the notation tractable.

Next, we simplify the system \eqref{eq:eq1}.
Permute the rows of \eqref{eq:eq1} to get the equivalent linear system
\begin{equation}
\setlength\arraycolsep{5pt}
0=
\begin{bmatrix}
*&1&0&0&0&0&0&0&0\\
*&*&*&1&0&0&0&0&0\\
*&*&*&*&*&1&*&*&*\\
0&-1&1&0&0&0&\alpha&0&0\\
0&0&0&-1&1&0&0&\beta&0\\
*&*&*&*&*&*&*&*&1
\end{bmatrix}
\begin{bmatrix}
z_0\\
z_1\\
x_1\\
z_2\\
x_2\\
Tz_0\\
\tilde{A}x_1\\
\tilde{B}x_2\\
Sz_0
\end{bmatrix}.
\label{eq:eq2}
\end{equation}
Since permuting the rows is a reversible process, \eqref{eq:eq1} and \eqref{eq:eq2} are equivalent.
In the step-by-step evaluation procedure of $(T,S)$, the evaluation of $T$ or $S$ completes before the other.
As the first case, assume the evaluation of $S$ completes first, which means the evaluation of $S$ does not depend on the evaluation of $T$.
Then the linear system is of the form
\[
\setlength\arraycolsep{5pt}
0=
\begin{bmatrix}
*&1&0&0&0&0&0&0&0\\
*&*&*&1&0&0&0&0&0\\
*&*&*&*&*&1&*&*&*\\
0&-1&1&0&0&0&\alpha&0&0\\
0&0&0&-1&1&0&0&\beta&0\\
*&*&*&*&*&\bm{0}&*&*&1
\end{bmatrix}
\begin{bmatrix}
z_0\\
z_1\\
x_1\\
z_2\\
x_2\\
Tz_0\\
\tilde{A}x_1\\
\tilde{B}x_2\\
Sz_0
\end{bmatrix}.
\]
The \textbf{boldface symbols} denote where to pay attention in the linear systems.
Perform Gaussian elimination to get
\[
\setlength\arraycolsep{5pt}
0=
\begin{bmatrix}
*&1&0&0&0&0&0&0&0\\
*&*&*&1&0&0&0&0&0\\
*&*&*&*&*&1&*&*&\bm{0}\\
0&-1&1&0&0&0&\alpha&0&0\\
0&0&0&-1&1&0&0&\beta&0\\
*&*&*&*&*&0&*&*&1
\end{bmatrix}
\begin{bmatrix}
z_0\\
z_1\\
x_1\\
z_2\\
x_2\\
Tz_0\\
\tilde{A}x_1\\
\tilde{B}x_2\\
Sz_0
\end{bmatrix}.
\]
This corresponds to left-multiplication by the invertible matrix
\[
\setlength\arraycolsep{3pt}
\begin{bmatrix}
1&0&0&0&0&0\\
0&1&0&0&0&0\\
0&0&1&0&0&*\\
0&0&0&1&0&0\\
0&0&0&0&1&0\\
0&0&0&0&0&1
\end{bmatrix}.
\]
As the other case, assume evaluation of $T$ completes first, which means the evaluation of $T$ does not depend on the evaluation of $S$. 
Then the linear system is of the form
\[
\setlength\arraycolsep{5pt}
0=
\begin{bmatrix}
*&1&0&0&0&0&0&0&0\\
*&*&*&1&0&0&0&0&0\\
*&*&*&*&*&1&*&*&\bm{0}\\
0&-1&1&0&0&0&\alpha&0&0\\
0&0&0&-1&1&0&0&\beta&0\\
*&*&*&*&*&*&*&*&1
\end{bmatrix}
\begin{bmatrix}
z_0\\
z_1\\
x_1\\
z_2\\
x_2\\
Tz_0\\
\tilde{A}x_1\\
\tilde{B}x_2\\
Sz_0
\end{bmatrix}.
\]
Perform Gaussian elimination to get
\begin{equation}
\setlength\arraycolsep{5pt}
0=
\begin{bmatrix}
*&1&0&0&0&0&0&0&0\\
*&*&*&1&0&0&0&0&0\\
*&*&*&*&*&1&*&*&0\\
0&-1&1&0&0&0&\alpha&0&0\\
0&0&0&-1&1&0&0&\beta&0\\
*&*&*&*&*&\bm{0}&*&*&1
\end{bmatrix}
\begin{bmatrix}
z_0\\
z_1\\
x_1\\
z_2\\
x_2\\
Tz_0\\
\tilde{A}x_1\\
\tilde{B}x_2\\
Sz_0
\end{bmatrix}.
\label{eq:eq3}
\end{equation}
This corresponds to left-multiplication by the invertible matrix
\[
\setlength\arraycolsep{5pt}
\begin{bmatrix}
1&0&0&0&0&0\\
0&1&0&0&0&0\\
0&0&1&0&0&0\\
0&0&0&1&0&0\\
0&0&0&0&1&0\\
0&0&*&0&0&1
\end{bmatrix}.
\]
Regardless of which of the two cases we start from, we arrive at the same linear system \eqref{eq:eq3}.
Continue the Gaussian elimination to get
\[
\setlength\arraycolsep{5pt}
0=
\begin{bmatrix}
*&1&0&0&0&0&0&0&0\\
*&\bm{0}&*&1&0&0&0&0&0\\
*&\bm{0}&*&\bm{0}&*&1&\bm{0}&\bm{0}&0\\
0&-1&1&0&0&0&\alpha&0&0\\
0&0&0&-1&1&0&0&\beta&0\\
*&\bm{0}&*&\bm{0}&*&0&\bm{0}&\bm{0}&1
\end{bmatrix}
\begin{bmatrix}
z_0\\
z_1\\
x_1\\
z_2\\
x_2\\
Tz_0\\
\tilde{A}x_1\\
\tilde{B}x_2\\
Sz_0
\end{bmatrix}.
\]
This corresponds to left-multiplying \eqref{eq:eq3} by the invertible matrix
\begingroup\makeatletter\def\f@size{7}\check@mathfonts
\[
\setlength\arraycolsep{3pt}
\begin{bmatrix}
1&0&0&0&0&0\\
*&1&0&0&0&0\\
*&0&1&0&0&0\\
0&0&0&1&0&0\\
0&0&0&0&1&0\\
*&0&0&0&0&1
\end{bmatrix}
\begin{bmatrix}
1&0&0&0&0&0\\
0&1&0&0&0&0\\
0&*&1&0&0&0\\
0&0&0&1&0&0\\
0&0&0&0&1&0\\
0&*&0&0&0&1
\end{bmatrix}
\begin{bmatrix}
1&0&0&0&0&0\\
0&1&0&0&0&0\\
0&0&1&0&0&0\\
0&0&0&1&0&0\\
0&0&0&0&1&0\\
0&0&0&*&*&1
\end{bmatrix}
\begin{bmatrix}
1&0&0&0&0&0\\
0&1&0&0&0&0\\
0&0&1&*&*&0\\
0&0&0&1&0&0\\
0&0&0&0&1&0\\
0&0&0&0&0&1
\end{bmatrix}.
\]\endgroup

We now explicitly parameterize the unspecified parameters one at a time.
\[
\setlength\arraycolsep{5pt}
0=
\begin{bmatrix}
\bm{-a}&1&0&0&0&0&0&0&0\\
*&0&*&1&0&0&0&0&0\\
*&0&*&0&*&1&0&0&0\\
0&-1&1&0&0&0&\alpha&0&0\\
0&0&0&-1&1&0&0&\beta&0\\
*&0&*&0&*&0&0&0&1
\end{bmatrix}
\begin{bmatrix}
z_0\\
z_1\\
x_1\\
z_2\\
x_2\\
Tz_0\\
\tilde{A}x_1\\
\tilde{B}x_2\\
Sz_0
\end{bmatrix}.
\]
The role of $a$ is to define $z_1=az_0$. So the evaluation of $(T,S)$ starts with $J_{\alpha A}(az_0)$.
If $a=0$, then $J_{\alpha A}$ ignores the input $z_0$ and always uses $0$ as the input.
Since $(T,S)$ accesses $A$ only through the evaluation of $J_{\alpha A}$,
how is it possible that $(T,S)$ evaluates $J_{\alpha A}$ only at $0$ and still encodes the zeros of $A+B$?

We now show $a\ne 0$.
Assume  $a=0$ for contradiction.
Let
\[
A(x)=c_1x, \qquad B(x)=c_2,
\]
where $c_1>0$ is unspecified and $0\ne c_2\in \reals^d$.
Then $J_{\alpha A}0=0$,
and
the mappings $T$ and $S$ are independent of the value of $c_1$.
So the set of fixed points of $T$ 
and the set of $Sz^\star$, where $z^\star$ is a fixed point of $T$,
do not depend on $c_1$.
However, the solution $\{-c_1^{-1}c_2\}=\zer(A+B)$ does depend on $c_1$.
Since $(T,S)$ is assumed to be a fixed-point encoding, $T$ must have a fixed-point $z^\star$ and it must satisfy $-c_1^{-1}c_2=Sz^\star$,
which is a contradiction.

Knowing $a\ne 0$, we can  absorb the top-left $a$ into $z_0$ and left-multiply by an invertible matrix to
get the equivalent system
\[
\setlength\arraycolsep{5pt}
0=
\begin{bmatrix}
1&0&0&0&0&0\\
0&1&0&0&0&0\\
0&0&\bm{a}&0&0&0\\
0&0&0&1&0&0\\
0&0&0 &0&1 &0\\
0&0&0&0&0&1
\end{bmatrix}
\begin{bmatrix}
\bm{-1}&1&0&0&0&0&0&0&0\\
*&0&*&1&0&0&0&0&0\\
*&0&* &0&* &1&0&0&0\\
0&-1&1&0&0&0&\alpha&0&0\\
0&0&0&-1&1&0&0&\beta&0\\
*&0&*&0&*&0&0&0&1
\end{bmatrix}
\begin{bmatrix}
\bm{a} z_0\\
z_1\\
x_1\\
z_2\\
x_2\\
Tz_0\\
\tilde{A}x_1\\
\tilde{B}x_2\\
Sz_0
\end{bmatrix}.
\]
This further simplifies to the equivalent system
\[
\setlength\arraycolsep{5pt}
0=
\begin{bmatrix}
-1&1&0&0&0&0&0&0&0\\
*&0&*&1&0&0&0&0&0\\
*&0&* &0&* &1&0&0&0\\
0&-1&1&0&0&0&\alpha&0&0\\
0&0&0&-1&1&0&0&\beta&0\\
*&0&*&0&*&0&0&0&1
\end{bmatrix}
\begin{bmatrix}
\bm{a} z_0\\
z_1\\
x_1\\
z_2\\
x_2\\
\bm{aT(a^{-1}az_0)}\\
\tilde{A}x_1\\
\tilde{B}x_2\\
\bm{S(a^{-1}az_0)}
\end{bmatrix}.
\]
By redefining $(T(z_0),S(z_0))$ to be the equivalent scaled splitting $(aT(a^{-1} a z_0),S(a^{-1}a z_0))$, we get
the equivalent system
\begin{equation}
\setlength\arraycolsep{5pt}
0=
\begin{bmatrix}
-1&1&0&0&0&0&0&0&0\\
\bm{\theta_1}&0&\bm{\theta_2}&1&0&0&0&0&0\\
\bm{\theta_3}&0&\bm{\theta_4 }&0&\bm{\theta_5 }&1&0&0&0\\
0&-1&1&0&0&0&\alpha&0&0\\
0&0&0&-1&1&0&0&\beta&0\\
\bm{\theta_6}&0&\bm{\theta_7}&0&\bm{\theta_8}&0&0&0&1
\end{bmatrix}
\begin{bmatrix}
\bm{z_0}\\
z_1\\
x_1\\
z_2\\
x_2\\
\bm{Tz_0}\\
\tilde{A}x_1\\
\tilde{B}x_2\\
\bm{Sz_0}
\end{bmatrix},
\label{eq:simplify}
\end{equation}
where we have now explicitly parameterized the remaining parameters as $\theta_1,\dots,\theta_8$.

The system \eqref{eq:simplify} defines $(T,S)$, i.e.,
it specifies the evaluation of $(T,S)$ at any input $z_0$.
Of course, $x_1$, $x_2$, and $Sz_0$ need not be solutions to the monotone inclusion problem, since the input $z_0$ is arbitrary.
To summarize our progress,
we have shown that
any frugal resolvent-splitting without lifting
is equivalent to a frugal resolvent-splitting 
of the form \eqref{eq:simplify}.

We now take a moment to consider what happens with DRS under this setup.
Although this discussion is not part of the proof, it will provide us with a sense of direction.
Under this formulation, DRS has the form
\begin{equation*}
\setlength\arraycolsep{5pt}
0=
\begin{bmatrix}
-1&1&0&0&0&0&0&0&0\\
1&0&-2&1&0&0&0&0&0\\
-1&0&\theta&0&-\theta&1&0&0&0\\
0&-1&1&0&0&0&\alpha&0&0\\
0&0&0&-1&1&0&0&\alpha&0\\
0&0&-1+\eta&0&-\eta&0&0&0&1
\end{bmatrix}
\begin{bmatrix}
z_0\\
z_1\\
x_1\\
z_2\\
x_2\\
Tz_0\\
\tilde{A}x_1\\
\tilde{B}x_2\\
Sz_0
\end{bmatrix}.
\end{equation*}
Row 1 defines $z_1=z_0$ as the input to $J_{\alpha A}$.
Row 4  corresponds to  $x_1=J_{\alpha A}z_1$.
Row 2 defines $z_2=2x_1-z_0$ as the input to $J_{\alpha B}$.
Row 5 corresponds to  $x_2=J_{\alpha B}z_2$.
Row 3 defines $Tz_0=z_0+\theta (x_2- x_1)$.
Row 6 defines $Sz_0=\eta x_2+(1-\eta)x_1$.
This linear system represents the evaluation of $(T,S)$ at any arbitrary input $z_0$,
i.e., the system defines $(T,S)$. 

To show that DRS is a fixed-point encoding, one considers evaluations of $(T,S)$ at fixed points
and shows
$x_1=x_2=Sz_0$ and $\tilde{A}x_1+\tilde{B}x_2=0$.
To do this, we add a row representing the fixed-point condition  $z_0=Tz_0$
\begin{equation*}
\setlength\arraycolsep{5pt}
0=
\begin{bmatrix}
-1&1&0&0&0&0&0&0&0\\
1&0&-2&1&0&0&0&0&0\\
-1&0&\theta&0&-\theta&1&0&0&0\\
\bm{-1}&\bm{0}&\bm{0 }&\bm{0}&\bm{0 }&\bm{1}&\bm{0}&\bm{0}&\bm{0}\\
0&-1&1&0&0&0&\alpha&0&0\\
0&0&0&-1&1&0&0&\alpha&0\\
0&0&-1+\eta&0&-\eta&0&0&0&1
\end{bmatrix}
\begin{bmatrix}
z_0\\
z_1\\
x_1\\
z_2\\
x_2\\
Tz_0\\
\tilde{A}x_1\\
\tilde{B}x_2\\
Sz_0
\end{bmatrix}.
\end{equation*}
Now the system represents evaluations of $(T,S)$ at fixed points.
We then left-multiply the system with
\[
(1/\theta)
\begin{bmatrix}
0&0&1 &-1&0&0&0
\end{bmatrix}
\]
to get $x_1=x_2$, left-multiply  the system with
\[
\begin{bmatrix}
0&0&-\eta/\theta &\eta/\theta&0&0&1
\end{bmatrix}
\]
to get $Sz_0=x_1$, and left-multiply  the system with
\[
(1/\alpha)\begin{bmatrix}
1&1&1/\theta &-1/\theta&1&1&0
\end{bmatrix}
\]
to get $0=\tilde{A}x_1+\tilde{B}x_2$.

With DRS, it is possible to perform Gaussian elimination with the linear equalities defining $(T,S)$
and the fixed-point condition $Tz_0=z_0$ to conclude $x_1=x_2=Sz_0$ and $\tilde{A}x_1+\tilde{B}x_2=0$.
With other fixed-point encodings, should we not be able to do the same?
How else could $Tz_0=z_0$ certify $Sz_0\in \zer(A+B)$?
This turns out to be true: we must be able to establish $x_1=x_2$, $x_1=Sz_0$, and $\tilde{A}x_1+\tilde{B}x_2=0$ through a linear combination of the linear equalities as otherwise we can construct counter examples that contradict the assumption that $(T,S)$ is a fixed-point encoding.


We now return to the proof.
Consider \eqref{eq:simplify} with the fixed-point condition $T(z_0)=z_0$ added
\begin{equation}
\setlength\arraycolsep{5pt}
0=\underbrace{
\begin{bmatrix}
-1&1&0&0&0&0&0&0&0\\
\theta_1&0&\theta_2&1&0&0&0&0&0\\
\theta_3&0&\theta_4 &0&\theta_5 &1&0&0&0\\
\bm{-1}&\bm{0}&\bm{0} &\bm{0}&\bm{0} &\bm{1}&\bm{0}&\bm{0}&\bm{0}\\
0&-1&1&0&0&0&\alpha&0&0\\
0&0&0&-1&1&0&0&\beta&0\\
\theta_6&0&\theta_7&0&\theta_8&0&0&0&1
\end{bmatrix}}
_{=M}
\underbrace{
\begin{bmatrix}
z_0\\
z_1\\
x_1\\
z_2\\
x_2\\
Tz_0\\
\tilde{A}x_1\\
\tilde{B}x_2\\
Sz_0
\end{bmatrix}}
_{=v}.
\label{eq:Mv_zero}
\end{equation}
System \eqref{eq:Mv_zero} represents evaluations of $(T,S)$ at a fixed points.

We claim that the linear equalities \eqref{eq:Mv_zero} must imply 
$x_1=x_2$, $Sz_0=x_1$, and $\tilde{A}x_1+\tilde{B}x_2=0$.
We prove these three implications one-by-one by assuming otherwise and constructing counter examples.


Assume for contradiction that \eqref{eq:Mv_zero} does not imply the  linear equality $x_1=x_2$.
By Lemma~\ref{lm-ge}, this means there is a specific instance
\[
v'=(z_0',
z_1',
x_1',
z_2',
x_2',
T(z_0'),
\tilde{A}x_1',
\tilde{B}x_2',
S(z_0'))
\in \reals^{9d}
\]
such that $Mv'=0$ but $x_1'\ne x_2'$.
The vector $v'$ represents an evaluation of\\
$(T(A,B,\cdot),S(A,B,\cdot))$ 
for any $A,B\in \mathcal{M}(\reals^d)$ satisfying
\[
\tilde{A}x_1' \in Ax_1',\quad
\tilde{B}x_2' \in B x_2'.
\]
The evaluation is at a fixed point, i.e., $T(A,B,z_0')=z_0'$, since we enforced $T(z_0)=z_0$ in \eqref{eq:Mv_zero}.
Define 
\[
A(x)=x-x_1'+\tilde{A}x_1',
\qquad
B(x)=x-x_2'+\tilde{B}x_2'.
\]
$A$ and $B$ are monotone operators constructed to match the evaluations
$A(x_1')=\tilde{A}x_1'$
and
$B(x_2')=\tilde{B}x_2'$.
Write $x^\star=S(A,B,z_0')$.
Since $(T,S)$ is a fixed-point encoding, we have
\[
0=(A+B)x^\star.
\]
However, $x_1'\ne x_2'$, so either $x_1'\ne x^\star$ or $x_2'\ne x^\star$ or both.
Without loss of generality assume  $x_1'\ne x^\star$.
Loosely speaking, $x_1\ne x^\star$ means $(T,S)$ was able to identify that $x^\star$ is a solution without examining the output of $A$ at $x^\star$, the purported solution.
Since the evaluation of $(T(A,B,z_0'),S(A,B,z_0'))$ depends on $A$ only through $Ax_1'$, 
what prevents us from changing the operator value at $x^\star$?
Define
\[
C(x)=2(x-x_1')+\tilde{A}x_1'.
\]
Since $C(x_1')=\tilde{A}x_1'$, we still have $T(C,B,z_0')=z_0'$ and $S(C,B,z_0')=x^\star$, i.e.,
changing $A$ to $C$ does not affect the evaluation of $T$ and $S$ at $z_0'$.
However, 
\[
0\notin (C+B)x^\star,
\]
since
$C(x^\star)\ne A(x^\star)$.
In other words, $T(C,B,z_0')=z_0'$, but $S(C,B,z_0')$ is not a zero of $C+B$.
So $(T,S)$ fails to be a fixed-point encoding for $C,B\in \mathcal{M}(\reals^d)$,  and we have a contradiction.
This proves that the linear system of equalities \eqref{eq:Mv_zero} does imply the linear equality $x_1=x_2$.
%
%

Next, assume for contradiction that 
\eqref{eq:Mv_zero} does not imply the linear equality $Sz_0=x_1$.
By Lemma~\ref{lm-ge}, this means there is a specific instance
\[
v'=(z_0',
z_1',
x_1',
z_2',
x_2',
T(z_0'),
\tilde{A}x_1',
\tilde{B}x_2',
S(z_0'))
\in \reals^{9d}
\]
such that $Mv'=0$ but $S(z_0')\ne x_1'=x_2' $.
(We now know that $x_1'=x_2' $.)
Again, loosely speaking, $S(z_0')\ne x_1'$ means $(T,S)$ was able to identify that $S(z_0')$ is a solution without examining the output of $A$ at $S(z_0')$, the purported solution, so we draw a contradiction by changing the operator value at $S(z_0')$.
Using the same definition of $A$, $B$, and $C$, 
the same arguments carry over and we can establish
$T(A,B,z_0')=T(C,B,z_0')=z_0'$ and 
$S(A,B,z_0')=S(C,B,z_0')$.
Define $x^\star=S(A,B,z_0')$. 
Since we assumed (for contradiction) that $x^\star\ne x_1'$, we have
\[
(A+B)(x^\star)\ne (C+B)(x^\star).
\]
Remember that $A$, $B$, and $C$ are single-valued.
So it is not possible for both $0=(A+B)(x^\star)$ and $0=(C+B)(x^\star)$ to be true.
Therefore $(T,S)$ fails to be a fixed-point encoding for the instance $A,B\in \mathcal{M}(\reals^d)$ or 
$C,B\in \mathcal{M}(\reals^d)$, and we have a contradiction.
This proves that the linear system of equalities \eqref{eq:Mv_zero} does imply the linear equality $Sz_0=x_1$.

Finally, assume for contradiction that 
\eqref{eq:Mv_zero} does not imply the linear equality $\tilde{A}x_1+\tilde{B}x_2=0$.
By Lemma~\ref{lm-ge}, this means there is a specific instance
\[
v'=(z_0',
z_1',
x_1',
z_2',
x_2',
T(z_0'),
\tilde{A}x_1',
\tilde{B}x_2',
S(z_0'))
\in \reals^{9d}
\]
such that $Mv'=0$ but $\tilde{A}x_1'+\tilde{B}x_2'\ne 0$.
Loosely speaking, $\tilde{A}x_1'+\tilde{B}x_2'\ne 0$ means $(T,S)$ was able to identify that $S(z_0')$ is a solution
without obtaining outputs of $A$ and $B$ that sum to $0$, and we draw a contradiction by demonstrating that
this is not possible when $A$ and $B$ are single-valued.
We now know that $x_1'=x_2'=S(z_0')$. 
Define 
\[
A(x)=x-x_1'+\tilde{A}x_1',
\qquad
B(x)=x-x_2'+\tilde{B}x_2'.
\]
Then $T(A,B,z_0)=z_0$, but
\[
(A+B)(S(z_0'))=
\tilde{A}x_1'+
\tilde{B}x_2'\ne 0,
\]
i.e., the purported solution $S(z_0')$ is not a solution.
So $(T,S)$ fails to be a fixed-point encoding for the instance $A,B\in\mathcal{M}(\reals^d)$,
and we have a contradiction.
This proves that the linear system of equalities \eqref{eq:Mv_zero} does imply the linear equality $\tilde{A}x_1+\tilde{B}x_2=0$.

With the assertions proved, we proceed to complete the proof.
Gaussian elimination on \eqref{eq:Mv_zero} gives us
the equivalent system
\begin{equation}
\setlength\arraycolsep{5pt}
0=
\begin{bmatrix}
-1&1&0&0&0&0&0&0&0\\
\theta_1&0&\theta_2&1&0&0&0&0&0\\
\bm{\theta_3+1}&0&\theta_4 &0&\theta_5 &\bm{0}&0&0&0\\
-1&0&0 &0&0 &1&0&0&0\\
0&-1&1&0&0&0&\alpha&0&0\\
0&0&0&-1&1&0&0&\beta&0\\
\theta_6&0&\theta_7&0&\theta_8&0&0&0&1
\end{bmatrix}
\begin{bmatrix}
z_0\\
z_1\\
x_1\\
z_2\\
x_2\\
Tz_0\\
\tilde{A}x_1\\
\tilde{B}x_2\\
Sz_0
\end{bmatrix}.
\label{eq:system-final}
\end{equation}
Because the system of linear equalities  must imply $x_1=x_2$
and because of where the zeros and nonzeros are placed,
we have
$\theta_3=-1$ and $\theta_4=-\theta_5=\theta$ for some $\theta\ne 0$.
Let us further spell out this argument.
The linear equality $x_1=x_2$ can be expressed as
\begin{equation}
\setlength\arraycolsep{5pt}
0=
\begin{bmatrix}
0&0&-1&0&1&0&0&0&0
\end{bmatrix}
\begin{bmatrix}
z_0\\
z_1\\
x_1\\
z_2\\
x_2\\
Tz_0\\
\tilde{A}x_1\\
\tilde{B}x_2\\
Sz_0
\end{bmatrix}.
\label{eq:system-result}
\end{equation}
By Lemma~\ref{lm-ge}, 
the system of linear equalities \eqref{eq:system-final} implies \eqref{eq:system-result}
if and only if we can linearly combine the rows of \eqref{eq:system-final}
to get \eqref{eq:system-result}.
Row $7$ of \eqref{eq:system-final} cannot be used in the linear combination,
as any nonzero contribution from row $7$ will place a nonzero component in the
$9$th column.
Row $6$ of \eqref{eq:system-final} also cannot be used in  the linear combination,
as any nonzero contribution from row $6$ will place a nonzero component in the
$8$th column.
Repeating this argument tells us that
rows $7$, $6$, $5$, $4$, $2$, and $1$ cannot be used in the linear combination.
Therefore, a scalar multiple of row $3$ of \eqref{eq:system-final} must equal \eqref{eq:system-result}, and this tells us
$\theta_3=-1$ and $\theta_4=-\theta_5=\theta$ for some $\theta\ne 0$.

Plugging in the values of $\theta_3$, $\theta_4$ and $\theta_5$, we get
\begin{equation}
\setlength\arraycolsep{5pt}
0=
\begin{bmatrix}
-1&1&0&0&0&0&0&0&0\\
\theta_1&0&\theta_2&1&0&0&0&0&0\\
0&0&\theta &0&-\theta &0&0&0&0\\
-1&0&0 &0&0 &1&0&0&0\\
0&-1&1&0&0&0&\alpha&0&0\\
0&0&0&-1&1&0&0&\beta&0\\
\theta_6&0&\theta_7&0&\theta_8&0&0&0&1
\end{bmatrix}
\begin{bmatrix}
z_0\\
z_1\\
x_1\\
z_2\\
x_2\\
Tz_0\\
\tilde{A}x_1\\
\tilde{B}x_2\\
Sz_0
\end{bmatrix}.
\label{eq:system-final2}
\end{equation}
Because the linear equalities must imply $x_1=Sz_0$
and because of where the zeros and nonzeros are placed,
$\theta_6=0$, $\theta_7=-1+\eta$, and $\theta_8=-\eta$ for some $\eta\in \reals$.
Let us further spell out this argument.
The linear equality $Sz_0=x_1$ can be expressed as
\begin{equation}
\setlength\arraycolsep{5pt}
0=
\begin{bmatrix}
0&0&-1&0&0&0&0&0&1
\end{bmatrix}
\begin{bmatrix}
z_0\\
z_1\\
x_1\\
z_2\\
x_2\\
Tz_0\\
\tilde{A}x_1\\
\tilde{B}x_2\\
Sz_0
\end{bmatrix}
\label{eq:theorem1-eq1}
\end{equation}
By Lemma~\ref{lm-ge}, the system of linear equalities \eqref{eq:system-final2} implies \eqref{eq:theorem1-eq1}
if and only if we can linearly combine the rows of \eqref{eq:system-final2}
to get \eqref{eq:theorem1-eq1}.
Row $6$ cannot be used in  the linear combination,
as any nonzero contribution will place a nonzero component in the $8$th column.
Row $5$ cannot be used in  the linear combination,
as any nonzero contribution will place a nonzero component in the $7$th column.
Repeating this argument tells us that rows $6$, $5$, $4$, $2$, and $1$ cannot be used in the linear combination.
This leaves us with the rows
\[
\setlength\arraycolsep{5pt}
0=
\begin{bmatrix}
0&0&\theta &0&-\theta &0&0&0&0\\
\theta_6&0&\theta_7&0&\theta_8&0&0&0&1
\end{bmatrix}
\begin{bmatrix}
z_0\\
z_1\\
x_1\\
z_2\\
x_2\\
Tz_0\\
\tilde{A}x_1\\
\tilde{B}x_2\\
Sz_0
\end{bmatrix}
\]
to imply \eqref{eq:theorem1-eq1}.
This is possible only if $\theta_6=0$, $\theta_7=-1+\eta$, and $\theta_8=-\eta$ for some $\eta\in \reals$.

Plugging in the values of $\theta_6$, $\theta_7$ and $\theta_8$, we get
\begin{equation}
\setlength\arraycolsep{5pt}
0=
\begin{bmatrix}
-1&1&0&0&0&0&0&0&0\\
\theta_1&0&\theta_2&1&0&0&0&0&0\\
0&0&\theta &0&-\theta &0&0&0&0\\
-1&0&0 &0&0 &1&0&0&0\\
0&-1&1&0&0&0&\alpha&0&0\\
0&0&0&-1&1&0&0&\beta&0\\
0&0&-1+\eta&0&-\eta &0&0&0&1
\end{bmatrix}
\begin{bmatrix}
z_0\\
z_1\\
x_1\\
z_2\\
x_2\\
Tz_0\\
\tilde{A}x_1\\
\tilde{B}x_2\\
Sz_0
\end{bmatrix}.
\label{eq:system-final3}
\end{equation}
Because the system of linear equalities must imply $0=\tilde{A}x_1+\tilde{B}x_2$
and because of where the zeros and nonzeros are placed,
we have
$\theta_1=\beta/\alpha$ and $\theta_2=-1-\beta/\alpha$.
Let us further spell out this argument.
The linear equality $0=\tilde{A}x_1+\tilde{B}x_2$ can be expressed as
\begin{equation}
\setlength\arraycolsep{5pt}
0=
\begin{bmatrix}
0&0&0&0&0&0&1&1&0
\end{bmatrix}
\begin{bmatrix}
z_0\\
z_1\\
x_1\\
z_2\\
x_2\\
Tz_0\\
\tilde{A}x_1\\
\tilde{B}x_2\\
Sz_0
\end{bmatrix}.
\label{eq:theorem1-eq2}
\end{equation}
Left-multiply \eqref{eq:system-final3} by the invertible matrix
\[
\setlength\arraycolsep{5pt}
\begin{bmatrix}
1&0&0&0&0&0&0\\
0&1&0&0&0&0&0\\
-(1+\theta_2)&1&1/\theta&0&-(1+\theta_2)&1&0\\
0&0&0&1&0&0&0\\
0&0&0&0&1&0&0\\
0&0&0&0&0&1&0\\
0&0&0&0&0&0&1
\end{bmatrix}
\]
to get
\begin{equation*}
\setlength\arraycolsep{4pt}
0=
\begin{bmatrix}
-1&1&0&0&0&0&0&0&0\\
\theta_1&0&\theta_2&1&0&0&0&0&0\\
\bm{1 + \theta_1 + \theta_2  }&\bm{0}&\bm{0} &\bm{0}& \bm{0}&\bm{0}&\bm{-(1 + \theta_2) \alpha  }&\bm{\beta }&\bm{0}\\
-1&0&0 &0&0 &1&0&0&0\\
0&-1&1&0&0&0&\alpha&0&0\\
0&0&0&-1&1&0&0&\beta&0\\
0&0&-1+\eta&0&-\eta&0&0&0&1
\end{bmatrix}
\begin{bmatrix}
z_0\\
z_1\\
x_1\\
z_2\\
x_2\\
Tz_0\\
\tilde{A}x_1\\
\tilde{B}x_2\\
Sz_0
\end{bmatrix}.
\end{equation*}
Left-multiply by the invertible matrix
\[
\setlength\arraycolsep{5pt}
\begin{bmatrix}
0&0&1&0&0&0&0\\
1&0&0&0&0&0&0\\
0&0&0&0&1&0&0\\
0&1&0&0&0&0&0\\
0&0&0&0&0&1&0\\
0&0&0&1&0&0&0\\
0&0&0&0&0&0&1
\end{bmatrix}
\]
to permute the rows and get
\begin{equation}
\setlength\arraycolsep{3.5pt}
0=
\begin{bmatrix}
1 + \theta_1 + \theta_2  &0&0 &0& 0&0&-(1 + \theta_2) \alpha  &\beta&0\\
-1&1&0&0&0&0&0&0&0\\
0&-1&1&0&0&0&\alpha&0&0\\
\theta_1&0&\theta_2&1&0&0&0&0&0\\
0&0&0&-1&1&0&0&\beta&0\\
-1&0&0 &0&0 &1&0&0&0\\
0&0&-1+\eta&0&-\eta&0&0&0&1
\end{bmatrix}
\begin{bmatrix}
z_0\\
z_1\\
x_1\\
z_2\\
x_2\\
Tz_0\\
\tilde{A}x_1\\
\tilde{B}x_2\\
Sz_0
\end{bmatrix}.
\label{eq:system-final4}
\end{equation}
By the lemma and the equivalence of \eqref{eq:system-final3} and \eqref{eq:system-final4},
the system of linear equalities \eqref{eq:system-final3} implies \eqref{eq:theorem1-eq2}
if and only if we can linearly combine the rows of \eqref{eq:system-final4} to get \eqref{eq:theorem1-eq2}.
Row $7$ cannot be used in the linear combination,
as any nonzero contribution will place a nonzero component in the $9$th column.
Row $6$ cannot be used in the linear combination,
as any nonzero contribution will place a nonzero component in the $6$th column.
Repeating this argument tells us that rows $7$, $6$, $5$, $4$, $3$, and $2$ cannot be used in the linear combination.
This leaves us with 
\[
\setlength\arraycolsep{4pt}
0=
\begin{bmatrix}
1 + \theta_1 + \theta_2  &0&0 &0& 0&0&-(1 + \theta_2) \alpha  &\beta&0\\
\end{bmatrix}
\begin{bmatrix}
z_0\\
z_1\\
x_1\\
z_2\\
x_2\\
Tz_0\\
\tilde{A}x_1\\
\tilde{B}x_2\\
Sz_0
\end{bmatrix}
\]
to imply \eqref{eq:theorem1-eq2} and this requires $\theta_1=\beta/\alpha$ and $\theta_2=-1-\beta/\alpha$.


Finally, plugging in the parameters and expressing the splitting in functional form, we get
the splitting of Theorem~\ref{thm:2op-necessary}.
\qed

\subsection{Proof of Theorem~\ref{thm:2op-sufficiency}}
When $\alpha=\beta$, the splitting $(T,S)$ of Theorem~\ref{thm:2op-necessary} reduces to the setup of DRS.
The fixed-point iteration with respect to the DRS operator converges for all maximal monotone $A$ and $B$ if and only if $\theta\in(0,2)$.
That DRS converges for $\theta\in(0,2)$ is well known \cite[\S26.3]{BauschkeCombettes2017_convex},
and that DRS may diverge for some maximal monotone operators when $\theta\notin(0,2)$
 can be verified by considering the operators $A=0$ and $B=N_{\{0\}}$, where 
$N_{\{0\}}$ denotes the normal cone operator with respect to the set $\{0\}$.

Now assume $\alpha\ne \beta$. 
We provide counter examples,
single-valued maximal monotone operators $A$ and $B$
such that $\{0\}=\zer(A+B)$ 
and $T^kz^0$ diverges for any $z^0\ne 0$.
Note that the parameters $\alpha$ and $\beta$ are fixed and are provided by the splitting.
Our counter examples rely on $\alpha$ and $\beta$.

For the moment, consider the case  $d=2$.
Consider the problem
\[
\underset{x\in \mathbb{R}^2}{\mbox{find}}\quad
0=(A+B)x,
\]
where
\[
A=
\setstackgap{L}{1\baselineskip}
\bracketMatrixstack{
0&\tan(\omega)/\alpha\\
-\tan(\omega)/\alpha&0
}
\qquad
B=
\setstackgap{L}{1\baselineskip}
\bracketMatrixstack{
0&-\tan(\omega)/\beta\\
\tan(\omega)/\beta&0
}
\]
and $\alpha,\beta>0$, and $\omega\in (0,\pi/2)$.
We identify $A$ and $B$ as maximal monotone operators from $\mathbb{R}^2\rightarrow\mathbb{R}^2$.
Note that  $x^\star=0$ is the unique solution.

With basic algebra, we can show that
\[
Tz
=
\setstackgap{L}{1\baselineskip}
\bracketMatrixstack{
1&{\color{red}\theta}(1-\beta/\alpha)\cos(\omega)\sin(\omega)\\
-{\color{red}\theta}(1-\beta/\alpha)\cos(\omega)\sin(\omega)&
1
}z.
\]
With basic eigenvalue computation, we get
\[
|\lambda_1|^2=|\lambda_2|^2
=
1+\left(
{\color{red}\theta}(1-\beta/\alpha)\cos(\omega)\sin(\omega)
\right)^2>1,
\]
where $\lambda_1,\lambda_2$ are the eigenvalues of the matrix that defines $T$.
So if $z^0\ne 0$, the iteration $z^{k+1}=Tz^k$ diverges in that $\|z^k\|\rightarrow \infty$
and $\|Sz^k\|\rightarrow\infty$.

When $d>2$, we arrive at the same conclusion with
\[
\begin{bmatrix}
A&0&0&\cdots&0\\
0&0&\ddots&\ddots&0\\
\vdots&\vdots&\ddots&\ddots&0\\
0&0&\cdots&&0
\end{bmatrix}\in \reals^{d\times d},
\qquad
\begin{bmatrix}
B&0&0&\cdots&0\\
0&0&\ddots&\ddots&0\\
\vdots&\vdots&\ddots&\ddots&0\\
0&0&\cdots&&0
\end{bmatrix}\in \reals^{d\times d},
\]
which is the same counter example
embedded into $d$ dimensions.
\qed

\section{Impossibility of 3 operator resolvent-splitting without lifting}
\label{s:impossibility}
Define the problem class \eqref{eq:3op} to be the collection of monotone 
inclusion problems of the form
\begin{equation}
\underset{x\in \mathbb{R}^d}{\mbox{find}}\quad
0\in (A+B+C)x
\tag{3op-$\reals^d$}
\label{eq:3op}
\end{equation}
with $A,B,C\in \mathcal{M}(\reals^d)$.
A pair of functions $(T,S)$ is a fixed-point encoding for the problem class \eqref{eq:3op} if
\[
\exists z^\star\in \reals^{d'} \text{ such that}\,
\left(
\begin{array}{ll}
T(A, B, C, z^\star)&=z^\star\\
S(A, B, C, z^\star)&=x^\star
\end{array}
\right)
\quad\Leftrightarrow\quad
0\in (A+B+C)(x^\star).
\]
We call
\[
T:\mathcal{M}(\reals^d)\times \mathcal{M}(\reals^d)\times \mathcal{M}(\reals^d)\times \reals^{d'}\rightarrow \reals^{d'}
\]
the fixed-point mapping
and
\[
S:\mathcal{M}(\reals^d)\times \mathcal{M}(\reals^d)\times \mathcal{M}(\reals^d)\times \reals^{d'}\rightarrow \reals^{d},
\]
the solution mapping.
The four key terms, resolvent-splitting, frugal, unconditional convergence, and no lifting,
are defined analogously.


To define the notion of resolvent-splitting without lifting 
for the problem class \eqref{eq:3op},
we define  the class of mappings $\mathcal{G}$
similarly to how we defined $\mathcal{F}$.
Let $I$ be the ``identity mapping'' defined as
$I:\mathcal{M}(\reals^d)\times\mathcal{M}(\reals^d)\times\mathcal{M}(\reals^d)\times\reals^d\rightarrow\reals^d$ 
and
$I(A,B,C,z)=z$ for any $A,B,C\in \mathcal{M}(\reals^d)$ and $z\in \reals^d$.
Let $J_{\alpha,1}$ be the resolvent with respect to the first operator
defined as
$J_{\alpha,1}:\mathcal{M}(\reals^d)\times\mathcal{M}(\reals^d)\times\mathcal{M}(\reals^d)\times\reals^d\rightarrow\reals^d$ 
and
$J_{\alpha,1 }(A,B,C,z)=J_{\alpha A}(z) $ for any $A,B,C\in \mathcal{M}(\reals^d)$ and $z\in \reals^d$.
Define  $J_{\beta,2}$ likewise with $J_{\beta,2 }(A,B,C,z)=J_{\beta B}(z) $
and  $J_{\gamma,3}$ likewise as $J_{\gamma,3 }(A,B,C,z)=J_{\gamma C}(z) $.
Let
\[
\mathcal{G}_0=\{I\}\cup
\{J_{\alpha ,1}\,|\,\alpha>0\}\cup
\{J_{\beta ,2}\,|\,\beta>0\}\cup
\{J_{\gamma,3}\,|\,\gamma>0\}.
\]
Recursively define
\[
\mathcal{G}_{i+1}=
\{F+G\,|\,F,G\in \mathcal{G}_{i}\}
\cup
\{F\circ G\,|\,F,G\in \mathcal{G}_{i}\}
\cup
\{\gamma F\,|\,F\in \mathcal{G}_i,\,\gamma\in \reals\}
\]
for $i=0,1,2,\dots$, where ``composition'' $F\circ G$ is defined analogously.
Finally, define
\[
\mathcal{G}=\bigcup^\infty_{i=0}\mathcal{G}_i.
\]
Elements of $\mathcal{G}$ 
map
$\mathcal{M}(\reals^d)\times\mathcal{M}(\reals^d)\times\mathcal{M}(\reals^d)\times\reals^d$
to $\reals^d$.
If $R\in \mathcal{F}$ and $A,B,C\in \mathcal{M}(\reals^d)$,
then $R(A,B,C,\cdot):\reals^d\rightarrow\reals^d$.
If $(T,S)$ is a fixed-point encoding for the problem class \eqref{eq:3op}
and $T,S\in \mathcal{G}$, then $(T,S)$ is a resolvent-splitting without lifting
for the problem class \eqref{eq:3op}.

%


Frugality is defined analogously with the notion of evaluation procedures.
We only use the notion of frugality informally for the problem class \eqref{eq:3op}.


Unconditional convergence is also defined analogously.
We say $(T,S)$ converges unconditionally for the problem class \eqref{eq:3op}
if
\[
T^kz^0 \rightarrow z^\star,\quad Sz^\star\in \zer(A+B+C)
\]
for any $z^0\in \reals^d$ and $A,B,C\in \mathcal{M}(\reals^d)$, when $\zer(A+B+C)\ne \emptyset$.

\subsection{Impossibility result}
If one could find a frugal, unconditionally convergent resolvent-splitting without lifting for \eqref{eq:3op},
it would be a satisfying generalization of DRS to 3 operators.
However, this is impossible. Even if we drop
frugality and convergence as requirements, this is impossible.

\begin{theorem}
\label{thm:3op-impossibility}
There is no resolvent-splitting \textbf{without lifting} for \eqref{eq:3op}.
\end{theorem}

\emph{Clarification.}
Assume $T(A,B,C,\cdot):\mathbb{R}^d\rightarrow\mathbb{R}^d$
and $S(A,B,C,\cdot):\mathbb{R}^d\rightarrow\mathbb{R}^d$
are constructed with finitely many resolvents,
\begin{align*}
J_{\alpha(1) A},J_{\alpha(2) A}, \dots, J_{\alpha(n_A) A}\\
J_{\beta(1) B},J_{\beta(2) B}, \dots, J_{\beta(n_B) B}\\
J_{\gamma(1) C},J_{\gamma(2) C}, \dots, J_{\gamma(n_C) C}
\end{align*}
where the parameters $\alpha(i)$, $\beta(j)$, $\gamma(k)$
may be different.
Theorem~\ref{thm:3op-impossibility} states that
$(T,S)$ fails to be a fixed-point encoding.

\emph{Clarification.}
Another way to state Theorem~\ref{thm:3op-impossibility}
is to say that
no element of the near-ring
$\mathcal{G}$ is a fixed-point encoding for \eqref{eq:3op}.

\subsection{Proof of Theorem~\ref{thm:3op-impossibility}}
\subsubsection{Outline}
The proof can be divided into in roughly three steps.
In the first step, we set up the notation and express the evaluation of $T$ with a set of linear and non-linear equalities.
In the second step, we show that the linear equalities, coupled with the fixed point condition and some additional assumptions,
cannot show that the three operators are evaluated at a same single point.
In the third step, we use the conclusion of the second step and a Farkas-type lemma 
to take a certain element from the null space of the linear system and use it to construct a counter example.

\subsubsection{Proof}
Assume for contradiction that $(T,S)$ is a resolvent-splitting without lifting.
Let $n$ be the total number of resolvent evaluations
required to compute $T$ and $S$.
The specific value of $n$ depends on how you count, i.e., whether you simplify things
and whether some resolvent evaluations are counted redundantly.
All that matters is that $n$ is finite.

Since $T,S\in \mathcal{G}$ there is a finite evaluation procedure for $(T,S)$,
and we can find a sequential ordering for the resolvent evaluations.
Using this ordering, 
we label the resolvents
$J_1,J_2,\dots,J_n$,
where $J_i$ is one of $J_{\alpha A}$, $J_{\beta B}$, or $J_{\gamma C}$ 
for some $\alpha>0$, $\beta>0$, or $\gamma>0$
for each $i=1,\dots,n$.
We call $z_i$ the point at which $J_i$ is evaluated and $x_i=J_i(z_i)$ for $i=1,\dots,n$.
In the process of evaluating $Tz_0$ and $Sz_0$,
we get $z_0,z_1,x_1,z_2,x_2,\dots,z_n,x_n$, in this order.
Since scalar multiplication and vector addition are the only operations allowed aside from resolvent evaluations, 
$z_i$ is defined as a linear combination of $z_0,z_1,x_1,z_2,x_2,\dots,z_{i-1},x_{i-1}$
for each $i=1,\dots,n$, by nature of the ordering.
Likewise, $Tz_0$ can be expressed as a linear combination
of $z_0,z_1,x_1,z_2,x_2,\dots,z_n,x_n$.
Without loss of generality, assume $J_{\alpha A}$, $J_{\beta B}$, and $J_{\gamma C}$ are all used least once with some $\alpha>0$,
$\beta>0$, and $\gamma>0$.
Otherwise, if, for example, $J_{\alpha A}$ is never used, 
we let $z_{n+1}=0$ and $J_{n+1}=J_{A}$ to fix the issue.
This is equivalent to evaluating the resolvent at the end and not using the output.

Say $J_{\alpha A}$, $J_{\beta B}$, and $J_{\gamma C}$ 
are evaluated $n_A$, $n_B$, and $n_C$ times, respectively.
So $n_A+n_B+n_C=n$.
Let $a(1),a(2),\dots,a(n_A)\in\{1,2,\dots,n\}$ be distinct indicies
and let
$\alpha (1),\alpha (2),\dots,\alpha (n_A)>0$ 
be parameters
so that
\[
x_{a(\ell)}=J_{a(\ell)}(z_{a(\ell)})=J_{\alpha(\ell) A}(z_{a(\ell)}).
\]
In other words, $x_{a(1)},x_{a(2)},\dots,x_{a(n_A)}$ are the outputs of 
the resolvents of $A$.
Likewise, let
$b(1),b(2),\dots,b(n_B)\in\{1,2,\dots,n\}$
and
$c(1),c(2),\dots,c(n_C)\in\{1,2,\dots,n\}$ be distinct indices
and let 
$\beta (1),\beta (2),\dots,\beta (n_B)>0$
and
$\gamma (1),\gamma (2),\dots,\gamma (n_C)>0$ 
be parameters so that
\[
x_{b(\ell)}=J_{b(\ell)}(z_{b(\ell)})=J_{\beta(\ell) B}(z_{b(\ell)})
\]
for $\ell=1,\dots,n_B$ and
\[
x_{c(\ell)}=J_{c(\ell)}(z_{c(\ell)})=J_{\gamma(\ell) C}(z_{c(\ell)})
\]
for $\ell=1,\dots,n_C$.

We express the evaluation of $Tz_0$
with the following system of linear and non-linear equalities:
\begin{align*}
0=&
\setlength\arraycolsep{5pt}
\begin{bmatrix}
*&1 & 0&0&0&0&\cdots &0&0&0\\
*&*&*&1&0&0&\cdots &0&0&0\\
*&*&*&*&*&1&\cdots &0&0&0\\
&&&\vdots&&&&\\
&&&\vdots&&&&\\
*&*&*&*&*&*&\cdots &1&0&0\\
*&*&*&*&*&*&\cdots &*&*&1
\end{bmatrix}
\begin{bmatrix}
z_0\\
z_1\\
x_1\\
z_2\\
x_2\\
z_3\\
\vdots\\
z_n\\
x_n\\
Tz_0
\end{bmatrix}
\\
&x_1=J_1(z_1), \,x_2=J_2(z_2),\,\dots,\,x_n=J_n(z_n),
\end{align*}
where the $*$ denote unspecified scalar coefficients.
(Each scalar in the matrix should be interpreted as a $d\times d$ block.
We have seen this notation in the 
proof of Theorem~\ref{thm:2op-necessary}.)
Each linear equality except the last one defines $z_i$ for $i=1,\dots,n$.
The last linear equality defines $T(z_0)$.
With Gaussian elimination, we obtain the simpler equivalent system
\begin{align}
0=&
\setlength\arraycolsep{5pt}
\begin{bmatrix}
*&1 & 0&0&0&0&\cdots &0&0&0\\
*&\bm{0}&*&1&0&0&\cdots &0&0&0\\
*&\bm{0}&*&\bm{0}&*&1&\cdots &0&0&0\\
&&&\vdots&&&&\\
&&&\vdots&&&&\\
*&\bm{0}&*&\bm{0}&*&\bm{0}&\cdots &1&0&0\\
*&\bm{0}&*&\bm{0}&*&\bm{0}&\cdots &\bm{0}&*&1
\end{bmatrix}
\begin{bmatrix}
z_0\\
z_1\\
x_1\\
z_2\\
x_2\\
z_3\\
\vdots\\
z_n\\
x_n\\
Tz_0
\end{bmatrix}
\nonumber
\\
&x_1=J_1(z_1), \,x_2=J_2(z_2),\,\dots,\,x_n=J_n(z_n).
\nonumber
\end{align}
The \textbf{boldface symbols} denote where to pay attention in the linear systems.
To summarize our progress, we have set up the notation and shown that 
these linear and non-linear equalities define the evaluation of $T$ at any input $z_0$.

We now take a moment to consider what happens with DRS under a similar formulation.
Although this discussion is not part of the proof, it will provide us with a sense of direction.
Under this formulation, DRS has the form
\begin{align*}
\setlength\arraycolsep{5pt}
0=&
\begin{bmatrix}
-1&1&0&0&0&0\\
1&0&-2&1&0&0\\
-1&0&\theta&0&-\theta&1
\end{bmatrix}
\begin{bmatrix}
z_0\\
z_1\\
x_1\\
z_2\\
x_2\\
Tz_0
\end{bmatrix}\\
&x_1=J_{\alpha A}(z_1), \,x_2=J_{\alpha B}(z_2).
\end{align*}
With DRS 
we can combine the linear equalities
with the fixed point condition $z_0=Tz_0$
to show that $x_1=x_2$ when the input $z_0$ is a fixed point.
More specifically, we add $z_0=Tz_0$ to the linear system
 \begin{align*}
\setlength\arraycolsep{5pt}
0=&
\begin{bmatrix}
-1&1&0&0&0&0\\
1&0&-2&1&0&0\\
\bm{-1}&\bm{0}&\bm{0}&\bm{0}&\bm{0}&\bm{1}\\
-1&0&\theta&0&-\theta&1
\end{bmatrix}
\begin{bmatrix}
z_0\\
z_1\\
x_1\\
z_2\\
x_2\\
Tz_0
\end{bmatrix}
\end{align*}
and left-multiply
\[
\setlength\arraycolsep{5pt}
(1/\theta)
\begin{bmatrix}
0&0&-1&1
\end{bmatrix}
\]
to get
 \[
 \setlength\arraycolsep{5pt}
 0=
\begin{bmatrix}
0&0&1&0&-1&0
\end{bmatrix}
\begin{bmatrix}
z_0\\
z_1\\
x_1\\
z_2\\
x_2\\
Tz_0
\end{bmatrix}.
\]
This is equivalent to combining
\begin{align*}
Tz_0&=z_0+\theta(x_2-x_1)\\
z_0&=Tz_0
\end{align*}
to conclude $x_1=x_2$ when the input $z_0$ is a fixed point.
Remember,  $x_1$ and $x_2$ are the points where $A$ and $B$ are (indirectly) evaluated.
($J_{\alpha A}$ is directly evaluated at $z_1$, so $z_1\in x_1+\alpha Ax_1$,
and we
indirectly obtain the output $(1/\alpha)(z_1-x_1)\in Ax_1$. Likewise, we 
indirectly obtain the output $(1/\alpha)(z_2-x_2)\in Bx_2$.)

These arguments show that with DRS, $T$ evaluates $A$ and $B$ at the same point when the input $z_0$ is a fixed point.
(Further arguments would establish that the same point is a solution by showing that the outputs of $A$ and $B$ sum to $0$.)
In general, given a fixed-point encoding $(T,S)$ and a fixed point $z_0$,
shouldn't the evaluation of $T$ examine the output of all (2 or 3) operators at the solution $Sz_0$?
Otherwise how could $z_0=Tz_0$ certify that $Sz_0$ is a solution?

We now return to the setup of \eqref{eq:3op}.
Can we combine  the linear equalities and the fixed-point condition $z_0=Tz_0$ to show that $A$, $B$, and $C$ are evaluated at a same single point? It turns out that we cannot. 
(This by itself is not a contradiction. Just because we can't show something with one approach doesn't mean it can't be shown.)
However, this approach runs into a problem. If we proceed to construct a counter example to draw a contradiction, we run into certain difficulties.
We need a modified approach.

Instead,  assume the input $z_0$  furthermore satisfies the additional linear equalities
\begin{align}
x_{a(1)}&=x_{a(2)}=\dots=x_{a(n_A)}\nonumber\\
x_{b(1)}&=x_{b(2)}=\dots=x_{b(n_B)}\label{e-xequal}\\
x_{c(1)}&=x_{c(2)}=\dots=x_{c(n_C)}\nonumber
\end{align}
in addition to the fixed-point condition $z_0=Tz_0$.
Since $(T,S)$ is a fixed point encoding, $z_0=Tz_0$ should certify that $Sz_0$ is a solution, regardless of the additional assumptions \eqref{e-xequal}.
Now can we use the linear equalities defining $T$, $z_0=Tz_0$, and \eqref{e-xequal} to show that $x_{a(1)}=x_{b(1)}=x_{c(1)}$?
No we cannot. Let's see why.

We add the fixed-point condition $z_0=Tz_0$ to the system of linear equalities and perform Gaussian elimination:
\begin{align}
0=&
\setlength\arraycolsep{5pt}
\underbrace{
\begin{bmatrix}
*&1 & 0&0&0&0&\cdots &0&0&0\\
*&0&*&1&0&0&\cdots &0&0&0\\
*&0&*&0&*&1&\cdots &0&0&0\\
&&&\vdots&&&&\\
&&&\vdots&&&&\\
*&0&*&0&*&0&\cdots &1&0&0\\
\bm{-1}&\bm{0}&\bm{0}&\bm{0}&\bm{0}&\bm{0}&\cdots &\bm{0}&\bm{0}&\bm{1}\\
\bm{*}&\bm{0}&\bm{*}&\bm{0}&\bm{*}&\bm{0}&\cdots &\bm{0}&\bm{*}&\bm{0}
\end{bmatrix}
}_{=M}
\underbrace{
\begin{bmatrix}
z_0\\
z_1\\
x_1\\
z_2\\
x_2\\
z_3\\
\vdots\\
z_n\\
x_n\\
Tz_0
\end{bmatrix}}_{=v}
\label{eq:3op-system}
\\
&x_1=J_1(z_1), \,x_2=J_2(z_2),\,\dots,\,x_n=J_n(z_n).
\nonumber
\end{align}
Define the last row of $M$ to be $m$.
Let $N$ be a matrix such that 
\begin{align*}
0=
N
\begin{bmatrix}
z_0\\
z_1\\
x_1\\
\vdots\\
z_n\\
x_n\\
Tz_0
\end{bmatrix}
\quad\Leftrightarrow\quad
\eqref{e-xequal}.
\end{align*}
More specifically, let $N\in \reals^{(n-3)d\times ({2n+2})d}$ contain
only $0$, $1$, and $-1$ and let the nonzeros only be on the columns corresponding to the $x$-variables (columns number $3,5,\dots,2n+1$).
Note that \eqref{e-xequal} represents $n-3$ linear equalities, and this is reflected as the number of rows in $N$.
The positions of the nonzeros in $N$ depend on the ordering of the resolvent evaluations, and $N$ is not unique.
Let
\[
L=\begin{bmatrix}
M\\N
\end{bmatrix},
\qquad
\tilde{N}=\begin{bmatrix}m\\N\end{bmatrix},
\]
where $m$ is the last row of $M$.
So $0=Lv$ means $v$ satisfies the linear equalities \eqref{eq:3op-system} and \eqref{e-xequal}.
Let's try to combine rows of $0=Lv$ to establish $x_{a(1)}=x_{b(1)}=x_{c(1)}$.
By Lemma~\ref{lm-ge}, $0=Lv$ implies $x_{a(1)}=x_{b(1)}$ and $x_{b(1)}=x_{c(1)}$
if and only if we can linearly combine the rows of $0=Lv$ to get
$x_{a(1)}=x_{b(1)}$ and $x_{b(1)}=x_{c(1)}$.


Every row of $M$ except the last one cannot be used in the linear combination to prove a linear equality only involving the $x$-variables,
as any nonzero contribution will place a nonzero component on a column corresponding to $z_i$ or $Tz_0$ (column number $1,2,4,6\dots,2n,2n+2$).
This leaves us with the rows of 
\[
0=\begin{bmatrix}m\\N\end{bmatrix}v
\]
  to show $x_{a(1)}=x_{b(1)}$ and $x_{b(1)}=x_{c(1)}$.
The linear equality $0=Nv$ enforces
\begin{align*}
x_{a(1)}&=x_{a(2)}=\dots=x_{a(n_A)}\nonumber\\
x_{b(1)}&=x_{b(2)}=\dots=x_{b(n_B)}\\
x_{c(1)}&=x_{c(2)}=\dots=x_{c(n_C)},\nonumber
\end{align*}
which are $n-3$ equalities.
Therefore, $\mathcal{N}(N)$, the nullspace of $N$, has codimension $(n-3)d$.
The linear equality $0=mv$ can establish $x_{a(1)}=x_{b(1)}$ or $x_{b(1)}=x_{c(1)}$, but not both.
If $0=\tilde{N}v$ implies both $x_{a(1)}=x_{b(1)}$ and $x_{b(1)}=x_{c(1)}$,
then in total $0=\tilde{N}v$ enforces $x_1=\dots=x_n$, which are $n-1$ equalities.
So $\mathcal{N}(\tilde{N})$, the nullspace of $\tilde{N}$,  would have codimension $(n-1)d$ or less, but this reduction in codimension by $2d$ is a contradiction since $m$ is just one row.

Therefore, the linear system $0=Lv$ does not imply both $x_{a(1)}=x_{b(1)}$ and $x_{b(1)}=x_{c(1)}$.
This by itself is not a contradiction.
Rather, we use this fact to construct a counter example, $A,B,C\in \mathcal{M}(\reals^d)$ such that $(T,S)$ fails to be a fixed-point encoding.
The additional assumption \eqref{e-xequal} will help us in this construction.

We construct the counter example for the case $d=1$.
When $d>1$, we can use the same $1$ dimensional construction repeated for the $d$ coordinates.
More specifically, if $A\in \mathcal{M}(\reals)$, then $\tilde{A}$ defined with
\[
\tilde{A}(x)=(A(x_1),A(x_2),\dots,A(x_d)),
\]
where $x=(x_1,x_2,\dots,x_d)\in \reals^d$, satisfies $\tilde{A}\in \mathcal{M}(\reals^d)$.
This sort of construction based on the $1$ dimensional counter example will provide a $d$ dimensional counter example.


Assume $d=1$.
By Lemma~\ref{lm-ge},
there is a specific instance
\[
v'=(z_0',z_1',x_1',z_2',x_2',\dots,z_n',x_n',T(z_0'))\in \reals^{2n+2}
\]
that satisfies $Mv'=0$ and the linear equalities of \eqref{e-xequal},
but
$x_{a(1)}'\ne x_{b(1)}'$ or
$x_{b(1)}'\ne x_{c(1)}'$ or both.
Without loss of generality, say
$x_{b(1)}'\ne x_{c(1)}'$.

Define $A$ such that
\[
J_{\alpha (i)A}(z_{a(i)}')=x_{a(1)}'
\]
for all $i=1,\dots,n_A$. In particular, we achieve this by defining
\[
A(x_{a(1)}')=
\left[\min_{i=1,\dots,n_A} \left(z_{a(i)}'-x_{a(1)}'\right)/\alpha(i),\max_{i=1,\dots,n_A} \left(z_{a(i)}'-x_{a(1)}'\right)/\alpha(i)\right].
\]
For the moment, leave $A(x)$ for $x\ne x_{a(1)}'$ unspecified.
Define $B(x_{b(1)}')$ and $C(x_{c(1)}')$ likewise.
By construction, $z_0'=T(A,B,C,z_0')$, even though $A$, $B$, and $C$ are not yet fully specified.
Write $x'=S(z_0')$. We have $x'\ne x_{b(1)}'$ or $x'\ne x_{c(1)}'$
since $x_{b(1)}'\ne x_{c(1)}'$.
Without loss of generality, let $x'\ne x_{c(1)}'$.

Now we define
\[
A(x)=
\left\{
\begin{array}{ll}
(x-x_{a(1)}')+\min\{A(x_{a(1)}')\}&\text{for }x<x_{a(1)}'\\
(x-x_{a(1)}')+\max\{A(x_{a(1)}')\}&\text{for }x>x_{a(1)}'
\end{array}
\right.
\]
and
\[
B(x)=
\left\{
\begin{array}{ll}
(x-x_{b(1)}')+\min\{{\color{red}B}(x_{b(1)}')\}&\text{for }x<x_{b(1)}'\\
(x-x_{b(1)}')+\max\{{\color{red}B}(x_{b(1)}')\}&\text{for }x>x_{b(1)}'.
\end{array}
\right.
\]
(This makes $A$ and $B$ maximal monotone.)
By construction, $(A+B)(x')$ is a bounded subset of $\reals$,
and $C(x')$ is unspecified. Depending on whether $x'<x_{c(1)}'$ or $x'>x_{c(1)}'$, 
we can make $C(x')$ an arbitrarily small or large value, respectively
(and still have $C$ be monotone).
In either case, we make $C(x')$ single-valued and so small or so large that 
$0\notin (A+B+C)(x')$.
We extend the definition of $C$ to all of $\reals$ to make it maximal monotone.

So we have maximal monotone operators $A$, $B$, and $C$,
such that $z_0'=T(A,B,C,z_0')$ but the $x'=S(z_0')$
does not satisfy $0\in (A+B+C)x'$.
This contradicts the assumption that $(T,S)$ is a fixed-point encoding.
\qed

\section{Attainment of 3 operator resolvent-splitting with minimal lifting}
\label{s:attainment}
Loosely speaking, we say $(\vT,S)$ is a a resolvent-splitting with \emph{$\ell$-fold lifting} for the problem class \eqref{eq:3op} 
if $(\vT,S)$ is a fixed-point encoding and
\[
\vT(A,B,C,\cdot):\reals^{\ell d}\rightarrow \reals^{\ell d},\qquad S(A,B,C,\cdot):\reals^{\ell d}\rightarrow \reals^{d}
\]
is constructed with scalar multiplication, vector addition, and resolvent evaluations.
Note that $1$-fold lifting corresponds to no lifting.
Frugality is defined analogously.
We define these terms informally since they are not used in a rigorous statement.
Theorem~\ref{thm:3op-impossibility} states a resolvent-splitting for \eqref{eq:3op}
requires lifting. Then how much?
The answer is $2$-fold lifting.

A standard trick to solve \eqref{eq:3op} is to ``copy'' variables and form an enlarged problem
\begin{equation*}
\underset{x_1,x_2,x_3\in \mathbb{R}^d}{\mbox{find}}\quad
0\in 
\begin{bmatrix}
Ax_1\\Bx_2\\Cx_3
\end{bmatrix}
 +N_{\{(x_1,x_2,x_3)\,|\,x_1=x_2=x_3\}}(x_1,x_2,x_3),
\end{equation*}
where $N_K$ is the normal cone operator with respect to the set $K$.
By applying DRS in an appropriately scaled space, we get
the parallel proximal algorithm (PPXA)  \cite{combettes2008,combettes2011}, which generalizes
Spingarn's method of partial inverse \cite{spingarn1985}.
The PPXA splitting is given by $(\vT,S)$
 \begin{align}
 x_A&=J_{(\gamma/\omega_A) A}(z_A)\nonumber\\
 x_B&=J_{(\gamma/\omega_B) B}(z_B)\nonumber\\
 x_C&=J_{(\gamma/\omega_C) C}(z_C)\nonumber\\
 \bar{z}&=\omega_Az_A+\omega_Bz_B+\omega_Cz_C\nonumber\\
 \bar{x}&=\omega_Ax_A+\omega_Bx_B+\omega_Cx_C\nonumber\\
 T_A(\vz)&=z_A+\theta(2\bar{x}-\bar{z}-x_A)\tag{PPXA}\\
 T_B(\vz)&=z_B+\theta(2\bar{x}-\bar{z}-x_B)\nonumber\\
 T_C(\vz)&=z_C+\theta(2\bar{x}-\bar{z}-x_C)\nonumber\\
S(\vz)&=J_{{\color{red}(\gamma/\omega_A)} A}(z_A),\nonumber
\end{align}
 where $\omega_A,\omega_B,\omega_C>0$ satisfy $\omega_A+\omega_B+\omega_C=1$ and $\theta\in(0,2)$.
This frugal, unconditionally convergent resolvent-splitting 
uses $3$-fold lifting, since
$\vT=(T_A,T_B,T_C):\reals^{3d}\rightarrow\reals^{3d}$.

So constructing a resolvent-splitting for \eqref{eq:3op}
is impossible with $1$-fold lifting,
but it is possible with $3$-fold lifting.
It turns out that $2$-fold lifting is sufficient,
and we therefore call $2$-fold lifting the \emph{minimal lifting} for \eqref{eq:3op}.

\subsection{Attainment result}
\begin{theorem}
\label{thm:attainment}
The pair $(\vT,S)$,
where $\vT:\reals^{2d}\rightarrow\reals^{2d}$
and $S:\reals^{2d}\rightarrow\reals^{d}$, defined as
\begin{align*}
x_1&=J_{\alpha A}(z_1)\nonumber\\
x_2&=J_{\alpha B}(x_1+z_2)\nonumber\\
x_3&=J_{\alpha C}\left(x_1-z_1+x_2-z_2\right)\nonumber\\
T_1(\vz)&=z_1+\theta(x_3- x_1)\nonumber\\
T_2(\vz)&=z_2+\theta(x_3- x_2)\nonumber\\
S(\vz)&=(1/3)(x_1+x_2+x_3)
\end{align*}
with $\vz=(z_1,z_2)$ and $\vT=(T_1,T_2)$,
is a fixed-point encoding,
and $(\vT,S)$
converges unconditionally for $\theta\in(0,1)$ and $\alpha>0$.
\end{theorem}
Therefore, $(\vT,S)$ for any $\theta\in (0,1)$ 
is a frugal, unconditionally convergent, resolvent-splitting with minimal lifting for \eqref{eq:3op}.
When $B=0$, the splitting of Theorem~\ref{thm:attainment} reduces to DRS.
In this sense, this splitting is a direct generalization of DRS with minimal lifting.

\emph{Remark.}
To the best of the author's knowledge, the splitting of
Theorem~\ref{thm:attainment} cannot be reduced to an instance of a known splitting method.
This is why Theorem~\ref{thm:attainment} is proved from first principles.

\subsection{Proof of Theorem~\ref{thm:attainment}}
Throughout the proof, write $\vy=(y_1,y_2)$ and $\vz=(z_1,z_2)$.
Without loss of generality, assume $\alpha=1$. 
We first show that $(\vT,S)$ is a fixed-point encoding.

Assume $\vz$ is a fixed point of $\vT$. Since $\vz$ is a fixed point, we have $T_1(\vz)=z_1$ and $T_2(\vz)=z_2$, and this implies $x_1=x_2=x_3$.
Write
\begin{align*}
a&=z_1-x_1\\
b&=x_1+z_2-x_2\\
c&=x_1-z_1+x_2-z_2-x_3.
\end{align*}
Add the three and use $x_1=x_3$ to get
\begin{align*}
a+b+c=0.
\end{align*}
Since $a\in Ax_1$,  $b\in Bx_2$, and $c\in Cx_3$,
by the definitions of $x_1$, $x_2$, and $x_3$, 
this proves $x_1=x_2=x_3$ is a solution to \eqref{eq:3op}.

Now assume $x^\star$ is a solution to \eqref{eq:3op}, 
and let 
$a\in Ax^\star$,  $b\in Bx^\star$, and $c\in Cx^\star$
so that $a+b+c=0$. 
We then define
\[
\vz^\star=(a+x^\star,b)
\]
It is straightforward to verify that $\vT(\vz^\star)=\vz^\star$ and $S(\vz^\star)=x^\star$.

Next we show that $(\vT,S)$ converges unconditionally for $\theta\in (0,1)$.
We show this by showing $\vT$ is nonexpansive for $\theta=1$,
and appealing to the KM iteration theorem \cite[Proposition 5.16]{BauschkeCombettes2017_convex}, which states that an averaged nonexpansive iteration converges to a fixed point, if a fixed point exists.

Let $\theta=1$.
Define $M:\reals^{2d}\rightarrow \reals^{2d}$ with
\[
M(x_1,x_2)=
\begin{bmatrix}
x_1+x_2\\
x_1+x_2
\end{bmatrix},
\] 
which is  the linear operator corresponding to the matrix
\[
\begin{bmatrix}
I& I\\
I& I
\end{bmatrix}\in \reals^{2d\times2d}.
\]
Define $\vU$ as
\[
\vU(\vz)=
\begin{bmatrix}
J_A(z_1)-z_1\\
J_B(z_2+x_1)-z_2
\end{bmatrix}\in \reals^{2d},
\]
and we can write
\[
\vT=-\vU+
\begin{bmatrix}
J_C\\
J_C
\end{bmatrix}
\circ
M
\circ \vU.
\]
To clarify, $\circ$ here denotes the composition of operators from $\reals^{2d}$ to $\reals^{2d}$.
Define
\[
N=\begin{bmatrix}
-I& I\\
I& -I
\end{bmatrix}\in \reals^{2d\times2d}.
\]
Then
\begin{align*}
\|\vT&(\vy)-\vT(\vz)\|^2\\
=&
\|\vU(\vy)-\vU(\vz)\|^2
+\left\|
\begin{bmatrix}
J_C\\
J_C
\end{bmatrix}
\circ
M
\circ \vU(\vy)
-
\begin{bmatrix}
J_C\\
J_C
\end{bmatrix}
\circ
M
\circ \vU (\vz)
\right\|^2\\
&-2
\left\langle
\vU(\vy)-\vU(\vz),
\begin{bmatrix}
J_C\\
J_C
\end{bmatrix}
\circ
M
\circ \vU(\vy)
-
\begin{bmatrix}
J_C\\
J_C
\end{bmatrix}
\circ
M
\circ \vU (\vz)
\right\rangle\\
\le
&\|\vU(\vy)-\vU(\vz)\|^2\\
&+\left<
M
\circ (\vU(\vy)-\vU(\vz)),
\begin{bmatrix}
J_C\\
J_C
\end{bmatrix}
\circ
M
\circ \vU(\vy)
-
\begin{bmatrix}
J_C\\
J_C
\end{bmatrix}
\circ
M
\circ \vU (\vz)
\right>\\
&-2
\left\langle
\vU(\vy)-\vU(\vz),
\begin{bmatrix}
J_C\\
J_C
\end{bmatrix}
\circ
M
\circ \vU(\vy)
-
\begin{bmatrix}
J_C\\
J_C
\end{bmatrix}
\circ
M
\circ \vU (\vz)
\right\rangle\\
=
&\|\vU(\vy)-\vU(\vz)\|^2\\
&+
(\vU(\vy)-\vU(\vz))^T
N
\left(
\begin{bmatrix}
J_C\\
J_C
\end{bmatrix}
\circ
M
\circ \vU(\vy)
-
\begin{bmatrix}
J_C\\
J_C
\end{bmatrix}
\circ
M
\circ \vU (\vz)
\right)\\
=
&\|\vU(\vy)-\vU(\vz)\|^2.
\end{align*}
The first line follows from simply plugging in the expression for $\vT$ and expanding the squares.
The second line, the inequality, follows from firm nonexpansiveness of $J_C$.
The third line follows from the reasoning
\[
v^T\begin{bmatrix}
I&I\\
I&I
\end{bmatrix}u-
2v^Tu=
v^T\begin{bmatrix}
-I&I\\
I&-I
\end{bmatrix}u=v^TNu
\]
for any $v,u\in\reals^{2d}$.
The last line follows from recognizing that the second term is $0$ with the reasoning
\[
\begin{bmatrix}
a\\b
\end{bmatrix}^T
N
\begin{bmatrix}
c\\c
\end{bmatrix}=
\begin{bmatrix}
a\\b
\end{bmatrix}^T
\begin{bmatrix}
-c+c\\c-c
\end{bmatrix}=0
\]
for any $a,b,c\in\reals^{d}$.
Next we have
\begin{align*}
\|\vU&(\vy)-\vU(\vz)\|^2\\
=&
\|\vy-\vz\|^2+\left\|
\begin{bmatrix}
J_A(y_1)-J_A(z_1)\\
J_B(y_2+J_A(y_1))-J_B(z_2+J_A(z_1))
\end{bmatrix}
\right\|^2\\
&-2\left<
\vy-\vz,
\begin{bmatrix}
J_A(y_1)-J_A(z_1)\\
J_B(y_2+J_A(y_1))-J_B(z_2+J_A(z_1))
\end{bmatrix}
\right>\\
= &
\|\vy-\vz\|^2
-\|J_A(y_1)-J_A(z_1)\|^2
-\|J_B(y_2+J_A(y_1))-J_B(z_2+J_A(z_1))\|^2\\
&
-
2\left(\left<y_1-z_1,J_A(y_1)-J_A(z_1)\right>
-\|J_A(y_1)-J_A(z_1)\|^2
\right)\\
&
-
2\left<y_2-z_2,J_B(y_2+J_A(y_1))-J_B(z_2+J_A(z_1))\right>\\
&+2\|J_B(y_2+J_A(y_1))-J_B(z_2+J_A(z_1))\|^2\\
\le &
\|\vy-\vz\|^2
-\|J_A(y_1)-J_A(z_1)\|^2
-\|J_B(y_2+J_A(y_1))-J_B(z_2+J_A(z_1))\|^2\\
&
-
2\left<y_2-z_2,J_B(y_2+J_A(y_1))-J_B(z_2+J_A(z_1))\right>\\
&+
2\left<y_2+J_A(y_1)-z_2-J_A(z_1),J_B(y_2+J_A(y_1))-J_B(z_2+J_A(z_1))\right>\\
= &
\|\vy-\vz\|^2
-\|J_A(y_1)-J_A(z_1)\|^2
-\|J_B(y_2+J_A(y_1))-J_B(z_2+J_A(z_1))\|^2\\
&+
2\left<J_A(y_1)-J_A(z_1),J_B(y_2+J_A(y_1))-J_B(z_2+J_A(z_1))\right>\\
 =&
 \|\vy-\vz\|^2
 -\|J_A(y_1)-J_A(z_1)
 \mathrel{\color{red}-}
 J_B(y_2+J_A(y_1))-J_B(z_2+J_A(z_1))\|^2\\
 \le&
 \|\vy-\vz\|^2.
\end{align*}
The first line follows from plugging in the definition of $\vU$ and expanding the squares.
The second line follows from separating the norm and inner product on $\reals^{2d}$ to separate norms and inner products on $\reals^d$.
The third line, the inequality, follows from applying the firm nonexpansiveness inequality twice, once for $J_A$ and once for $J_B$.
The fourth line follows from combining the two inner products.
The fifth line follows from completing the square. 
The final inequality follows from dropping the negative sum of square.
\qed

\subsection{Numerical examples}
Whether the splitting of Theorem~\ref{thm:attainment} is fast or efficient is somewhat beside the point.
The purpose of Theorem~\ref{thm:attainment} is to establish attainment of minimal lifting, and it says nothing about the rate of convergence.

Nevertheless, we present some experiments with the splitting of Theorem~\ref{thm:attainment} in this section.
These experiments are meant to be merely illustrative, and whether the splitting of Theorem~\ref{thm:attainment} has any advantage over existing methods such as PPXA and whether the notion of minimal lifting translates to any practical performance advantage is a question to be addressed in future work.


\paragraph{Signal denoising with outliers.}
Consider the problem
\[
\begin{array}{ll}
\underset{x\in \reals^d}{\mbox{minimize}}&
\|x_S-a\|_1+\lambda \|Ux-b\|_1\\
\mbox{subject to}& x\ge 0,
\end{array}
\]
where $S\subseteq \{1,2,\dots,d\}$, $a\in \reals^{|S|}$, $b\in \reals^d$, and $U\in \reals^{d\times d}$
is a unitary matrix representing a wavelet transform.
The statistical interpretation is that we noisily observe $x$ on a subset $S$ of its indices, noisily observe $x$ in the wavelet domain,
and have a priori knowledge that $x$ is nonnegative.
The $\ell^1$-norm is used for robustness against outliers.
We reformulate this problem as
\[
\underset{x\in \reals^d}{\mbox{minimize}}
\quad
\underbrace{\|x_S-a\|_1}_{=f(x)}
+
\underbrace{\lambda \|Ux-b\|_1}_{=g(x)}
+
\underbrace{\delta_{\reals_+^d}(x)}_{=h(x)}
\]
and apply the splitting of Theorem~\ref{thm:attainment}, PPXA, and PDHG
with $A=\partial f$, $B=\partial g$, and $C=\partial h$.
Because $U$ is unitary, $J_{\alpha \partial g}$ has a closed-form formula.
For the experiments, we used synthetic data with  $d = 2^{20}$
and $|S|\approx d/5$.
The code for data generation and optimization is provided on the author's website for scientific reproducibility.

Figure~\ref{fig:denoising} shows the results.
The splitting of Theorem~\ref{thm:attainment}, which uses $2$-fold lifting, is competitive with PPXA and PDHG, which use $3$-fold lifting.
For all methods, the parameters were roughly tuned for best performance.
We do not plot distance to solution, since solution does not seem to be unique.

\begin{figure}[h]
\begin{center}
\includegraphics[width=.49\textwidth]{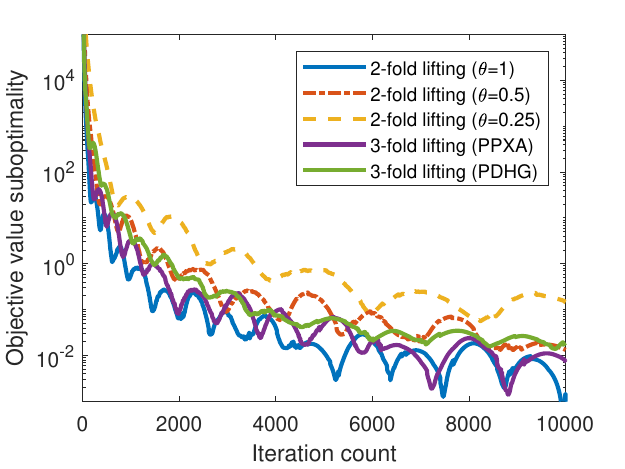}
\includegraphics[width=.49\textwidth]{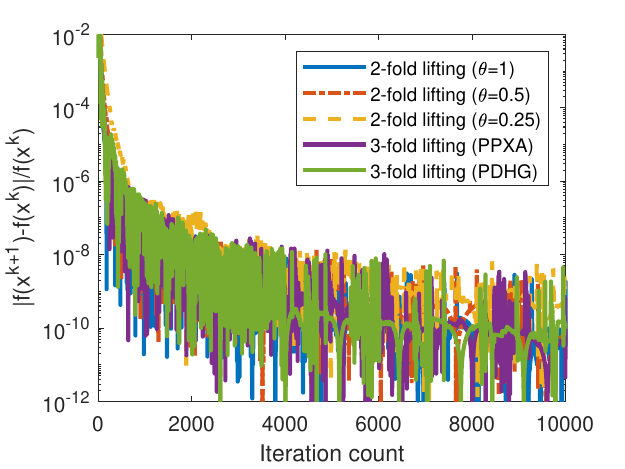}
\caption{Objective value and $|f(x^{k+1})-f(x^{k})|/f(x^{k})$ vs.\ iterations for the denoising problem.
}
\label{fig:denoising}
\end{center}
\end{figure}

%

\begin{figure}
\begin{center}
\includegraphics[width=.49\textwidth]{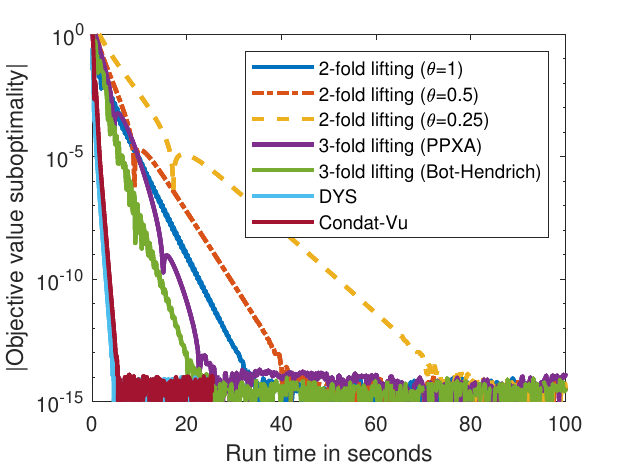}
\includegraphics[width=.49\textwidth]{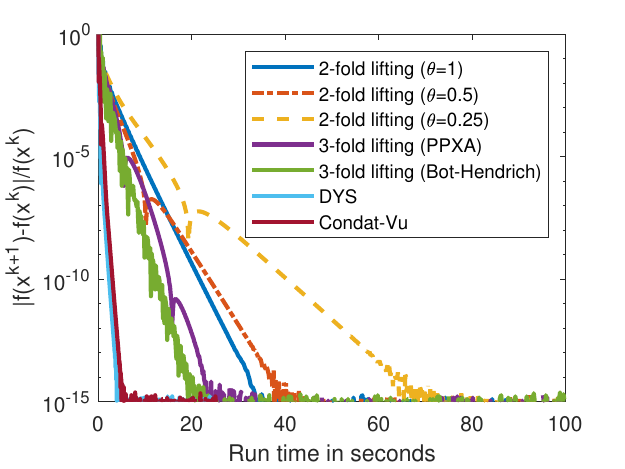}
\includegraphics[width=.49\textwidth]{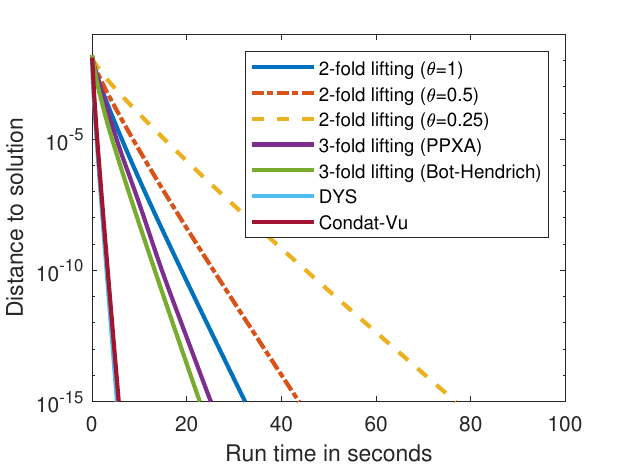}
\caption{$|$Objective value suboptimality$|$, $|f(x^{k+1})-f(x^{k})|/f(x^{k})$, and distance to solution vs.\ iterations for the portfolio optimization problem.
We take the absolute value  in the first plot,
 because the slightly infeasible iterates produce objective values lower than the optimal value.
The rough cost per iteration is $0.025s$ for CV and DYS
and $0.15s$ for  the splitting of Theorem~\ref{thm:attainment}, PPXA, and Bot--Hendrich.
}
\label{fig:portfolio}
\end{center}
\end{figure}

\paragraph{Portfolio optimization.}
Consider the Markowitz portfolio optimization \cite{brodie2009} problem
\[
\begin{array}{ll}
\underset{x\in \reals^d}{\mbox{minimize}}&
(1/2)\sum^n_{i=1}(a_i^Tx-b)^2\\
\mbox{subject to}&x\in \mathrm{\Delta}\\
&
\mu^Tx\ge b,
\end{array}
\]
where $d$ is the number of assets, $a_1,\dots,a_n\in \reals^d$ are $n$ realizations of the returns on the assets,
$\mathrm{\Delta}=\{x\in \reals^d\,|\,x_1,\dots,x_d\ge 0,\,x_1+\dots+x_d=1\}$
is the standard simplex for portfolios with no short positions,
$\mu\in \reals^d$ is the (estimated) average return of the assets,
and $b\in \reals$ is the desired expected return.
We reformulate this problem as
\[
\underset{x\in \reals^d}{\mbox{minimize}}
\quad
\underbrace{\frac{1}{2}\sum^n_{i=1}(a_i^Tx-b)^2}_{=f(x)}
+
\underbrace{\delta_\mathrm{\Delta}(x)}_{=g(x)}
+
\underbrace{\delta_{\{x\,|\,\mu^Tx\ge b\}}(x)}_{=h(x)}
\]
and apply the splitting of Theorem~\ref{thm:attainment} and PPXA with $A=\nabla f$, $B=\partial g$, and $C=\partial h$.
We also run the method of Bo\c{t} and Hendrich \cite[Algorithm 3.1]{bot_primal-dual_2013}, which has been used to solve portfolio optimization problems \cite{bot_portfolio2015}.
To evaluate $J_{\alpha\nabla f}$, we compute the Cholesky factorization of $(I+\alpha A^TA)$ once
and use the direct formula
\[
J_{\alpha\nabla f}(z)=
(I+\alpha A^TA)^{-1}(z+\alpha A^Tb\mathbf{1})
\]
where
\[
A=\begin{bmatrix}
a_1^T\\
\vdots\\
a_n^T
\end{bmatrix},\qquad
\mathbf{1}=\begin{bmatrix}
1\\
\vdots\\
1
\end{bmatrix}.
\]
Finally, we also run DYS \cite{davis2017} and Condat--V\~u \cite{condat2013,vu2013},
for which direct evaluations of $\nabla f$ were used  instead of $J_{\alpha\nabla f}$.
To compute the projection onto the simplex, we use the algorithm and code of
\cite{chen2011}.
For the experiments, we used synthetic data with $n= 30000$ and $d = 10000$,
which make the data approximately 2GB in size.
The Cholesky factorization of $I+\alpha A^TA$ requires about 30 minutes to compute for this problem size.
The code for data generation and optimization is provided on the author's website for scientific reproducibility.

Figure~\ref{fig:portfolio} shows the results.
The splitting of Theorem~\ref{thm:attainment}, which uses $2$-fold lifting, is competitive with PPXA and Bo\c{t}--Hendrich, which use $3$-fold lifting.
However, DYS and Condat--V\~u are faster than the splittings that only use resolvents.
Run-time measurements exclude the time it took to compute the Cholesky factorization, about $35.9s$.
An Intel Core i7-2600 CPU operating at 3.40GHz was used for the experiments.
For all methods, the parameters were roughly tuned for best performance.


\begin{figure}
\begin{center}
\includegraphics[width=.49\textwidth]{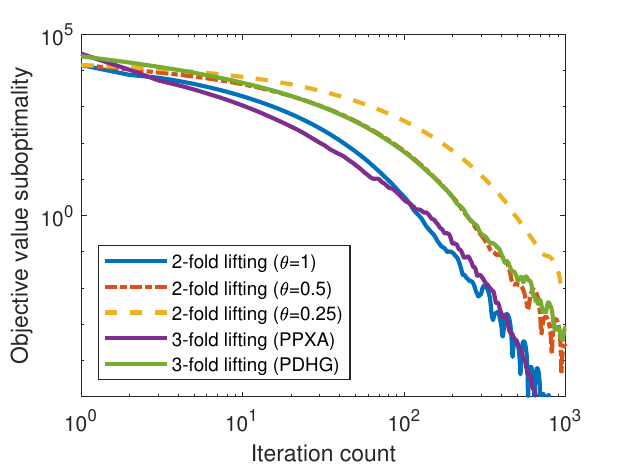}
\includegraphics[width=.49\textwidth]{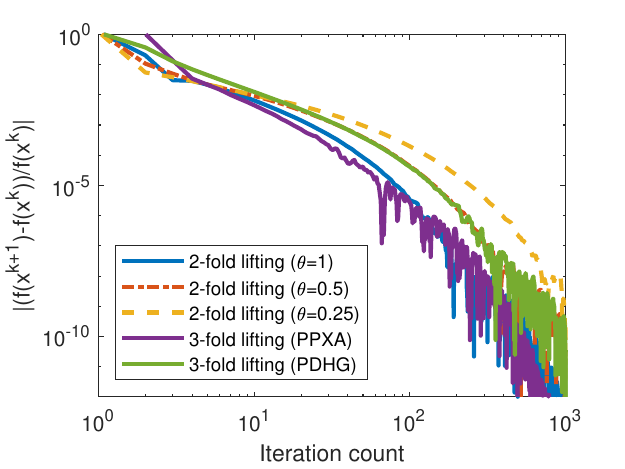}
\includegraphics[width=.49\textwidth]{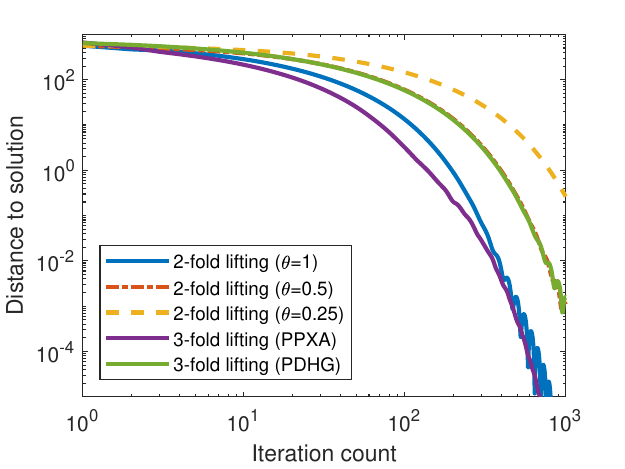}
\caption{Objective value, $|(f(x^{k+1})-f(x^{k}))/f(x^{k})|$, and distance to solution vs.\ iterations for the Poisson denoising with 1D total variation problem.
}
\label{fig:fused_lasso}
\end{center}
\end{figure}

\paragraph{Poisson denoising with 1D total variation}
Consider the problem
\[
\begin{array}{ll}
\underset{x\in \reals^d}{\mbox{minimize}}&
\lambda
\sum^d_{i=1}\ell(x_i;y_i)
+\sum_{i=1}^{d-1}
|x_{i+1}-x_i|
\end{array}
\]
where $y\in \mathbb{R}^d$, $\lambda>0$, and
\[
\ell(x;y)=\left\{
\begin{array}{ll}
-y^T\log (x)+x&\text{ for }y>0,\,x>0\\
0&\text{ for }y=0,\,x\ge 0\\
\infty&\text{ otherwise.}
\end{array}
\right.
\]
The statistical interpretation is that we wish to recover a 1D signal with small total variation corrupted by Poisson noise.
The first term is the negative log-likelihood for Poisson random variables \cite{byrne_poisson_1993,le_poisson_2007,chaux_poisson_2009,zanella_poisson_2009} and the second term is the 1D total variation penalty, also called fused lasso in the statistics literature \cite{tibshirani2005,rapaport2008,tibshirani2008}.
1D total variation denoising has been studied in \cite{Barbero:2011:FNM:3104482.3104522,5981401,6288032,WAHLBERG201283,condat2013_denoising}.
For simplicity, assume $d$ is odd.
We reformulate this problem as
\[
\underset{x\in \reals^d}{\mbox{minimize}}
\quad
\underbrace{\lambda\sum^d_{i=1}\ell(x_i;y_i)
}_{=f(x)}
+
\underbrace{\sum_{i=1,3,\dots,d-2}
|x_{i+1}-x_i|}_{=g(x)}
+
\underbrace{\sum_{i=2,4,\dots,d-1}
|x_{i+1}-x_i|}_{=h(x)}
\]
and apply the splitting of Theorem~\ref{thm:attainment}, PPXA, and PDHG
with $A=\nabla f$, $B=\partial g$, and $C=\partial h$.
We compute $J_{\alpha\nabla f}$ with 
\[
J_{\alpha\nabla f}(z)=\frac{1}{2}
\left(z-\alpha\lambda +\sqrt{(z-\alpha\lambda)^2+4\alpha\lambda y}\right),
\]
where the operations are elementwise.
Although $f$ is differentiable, its domain is not closed and the gradient is not Lipschitz continuous.
Therefore, splittings that use $\nabla f$ are not applicable, unless a line-search is implemented.
For the experiments, we used synthetic data with $n= 3000$, $d = 1001$, and $\lambda=1$. 
The code for data generation and optimization is provided on the author's website for scientific reproducibility.

Figure~\ref{fig:fused_lasso} shows the results.
The splitting of Theorem~\ref{thm:attainment}, which uses $2$-fold lifting, is competitive with PPXA and PDHG, which use $3$-fold lifting.
For all methods, the parameters were roughly tuned for best performance.

\section{Conclusion}
This work establishes that 
DRS is the unique frugal, unconditionally convergent resolvent-splitting without lifting 
for the 2 operator problem
and that 
there is no resolvent-splitting without lifting for the 3 operator problem.
Furthermore, this work presents a 
novel, frugal, unconditionally convergent resolvent-splitting
for the 3 operator problem
that directly generalizes DRS.
This splitting proves that 2-fold lifting is the minimal lifting
necessary for the 3 operator problem.
In other words, the presented splitting is optimal
in terms of frugality and lifting.

The potential for future work based on the ideas presented in this work is large. Analyzing and establishing uniqueness or optimality
of other splittings is one direction of future work. Characterizing all splittings of a given setup is another.
In particular, there is no reason to believe
the splitting of Theorem~\ref{thm:attainment} is unique, so characterizing all
frugal, unconditionally convergent resolvent-splittings for the 3 operator problem
would be interesting.

\begin{acknowledgements}
I would like to thank Wotao Yin for helpful comments and suggestions.
I would also like to thank the anonymous associate editor and referees whose comments improved the paper significantly.
In particular, the signal denoising numerical example was suggested by one of the anonymous reviewers.
This work is supported in part by NSF grant DMS-1720237 and ONR grant
N000141712162.
\end{acknowledgements}

\bibliographystyle{spmpsci}      
\bibliography{3op}   

@article{lions1979,
author = {P. L. Lions and B. Mercier},
title = {Splitting Algorithms for the Sum of Two Nonlinear Operators},
journal = {SIAM J. Numer. Anal.},
volume = {16},
number = {6},
pages = {964--979},
year = {1979},
}

@article{peaceman1955,
     title = {The Numerical Solution of Parabolic and Elliptic Differential Equations},
     author = {D. W. Peaceman and H. H. Rachford},
     journal = {J. Soc. Ind. Appl. Math.},
     volume = {3},
     number = {1},
     year = {1955},
     pages = {28--41},
}

@article{douglas1956,
author = {J. Douglas and H. H. Rachford},
title = {On the numerical solution of heat conduction problems in two and three space variables},
journal={Trans. Amer. Math. Soc.},
volume = {82},
pages = {421--439},
year = {1956}
}

@article{davis2017,
author="Davis, Damek and Yin, Wotao",
title="A Three-Operator Splitting Scheme and its Optimization Applications",
journal="Set-Valued Var. Anal.",
year="2017",
volume = {25},
number = {4},
pages = {829--858},
}

@book{BauschkeCombettes2017_convex,
  title = {Convex Analysis and Monotone Operator Theory in Hilbert Spaces},
  publisher = {{Springer New York}},
  author = {Bauschke, Heinz H. and Combettes, Patrick L.},
  edition = {2nd},
  year = {2017}
}

@article{ryu2016,
author = {Ernest K. Ryu and Stephen Boyd},
title = {Primer on Monotone Operator Methods},
journal  = {Appl. Comput. Math.},
volume = {15},
issue = {1},
pages = {3--43},
year = {2016},
}

@article{passty1979,
title = "Ergodic convergence to a zero of the sum of monotone operators in {H}ilbert space",
journal = "J. Math. Anal. Appl.",
volume = "72",
number = "2",
pages = "383--390",
year = "1979",
author = "G. B. Passty"
}

@article{esser2010,
author = {Ernie Esser and Xiaoqun Zhang and Tony F. Chan},
title = {A General Framework for a Class of First Order Primal-Dual Algorithms for Convex Optimization in Imaging Science},
journal = {SIAM J. Imaging Sci.},
volume = {3},
number = {4},
pages = {1015--1046},
year = {2010},
}

@article{ChaPoc11,
  title={A first-order primal-dual algorithm for convex problems with applications to imaging},
  author={Chambolle, Antonin and Pock, Thomas},
  journal={J. Math. Imaging Vis.},
  volume={40},
  number={1},
  pages={120--145},
  year={2011},
  publisher={Springer}
}

@article{zhu08,
  title = {An Efficient Primal-Dual Hybrid Gradient Algorithm For Total Variation Image Restoration},
  author = {Mingqiang Zhu and Tony Chan},
  journal = {UCLA CAM Report 08-34},
  year = {2008},
}

@article{pock2009,
title = {An algorithm for minimizing the Mumford-Shah functional},
year = {2009},
author = {Thomas Pock and Daniel Cremers and Horst Bischof and Antonin Chambolle},
journal = {IEEE Int. Conf. Comput. Vis.},
}

@Article{condat2013,
author="Condat, Laurent",
title="A Primal--Dual Splitting Method for Convex Optimization Involving {L}ipschitzian, Proximable and Linear Composite Terms",
journal="J. Optim. Theory Appl.",
year="2013",
volume="158",
number="2",
pages="460--479",
}

@Article{vu2013,
author={B. C. V{\~{u}}},
title="A splitting algorithm for dual monotone inclusions involving cocoercive operators",
journal="Adv. Comput. Math.",
year="2013",
volume="38",
number="3",
pages="667--681",
}

@Article{yan2016,
author="Yan, Ming",
title="A New Primal--Dual Algorithm for Minimizing the Sum of Three Functions with a Linear Operator",
journal="Journal of Scientific Computing",
year="2018",
volume="76",
number="3",
pages="1698--1717",
}

@article{tseng2000,
author = {P. Tseng},
title = {A Modified Forward-Backward Splitting Method for Maximal Monotone Mappings},
journal = {SIAM J. Control Optim.},
volume = {38},
number = {2},
pages = {431--446},
year = {2000},
}

@article{ba2017,
title = {Forward-Backward-Half Forward Algorithm with non Self-Adjoint Linear Operators for Solving Monotone Inclusions},
author = {Luis M. Brice{\~n}o-Arias and Damek Davis},
year = {2018},
journal = {SIAM Journal on Optimization}
}

@article{loris2011,
  author={Ignace Loris and Caroline Verhoeven},
  title={On a generalization of the iterative soft-thresholding algorithm for the case of non-separable penalty},
  journal={Inverse Problems},
  volume={27},
  number={12},
  year={2011},
  pages = {125007},
}

@article{chen2013,
  author={Peijun Chen and Jianguo Huang and Xiaoqun Zhang},
  title={A primal-dual fixed point algorithm for convex separable minimization with applications to image restoration},
  journal={Inverse Problems},
  volume={29},
  number={2},
  pages={025011},
  year={2013},
}

@article{driori2015,
title = "A simple algorithm for a class of nonsmooth convex-concave saddle-point problems",
journal = "Oper. Res. Lett.",
volume = "43",
number = "2",
pages = "209--214",
year = "2015",
author = "Yoel Drori and Shoham Sabach and Marc Teboulle",
}

@Article{Chen2016,
author="Chen, Peijun
and Huang, Jianguo
and Zhang, Xiaoqun",
title="A primal-dual fixed point algorithm for minimization of the sum of three convex separable functions",
journal="Fixed Point Theory and Applications",
year="2016",
}

@Article{latafat2017,
author="Latafat, Puya
and Patrinos, Panagiotis",
title="Asymmetric forward--backward--adjoint splitting for solving monotone inclusions involving three operators",
journal="Comput. Optim. Appl.",
year="2017",
volume="68",
number="1",
pages="57--93",
}

@article{ba2011,
author = {Luis M. Brice{\~n}o-Arias and Patrick L. Combettes},
title = {A Monotone+Skew Splitting Model for Composite Monotone Inclusions in Duality},
journal = {SIAM J. Optim.},
volume = {21},
number = {4},
pages = {1230--1250},
year = {2011},
}

@article{combettes2014,
author = {P. L. Combettes and L. Condat and J.-C. Pesquet and B. C. V{\~u}},
title = {A forward-backward view of some primal-dual optimization methods in image recovery},
year = {2014},
journal = {IEEE Int. Conf. Image Process.}
}

@article{martinet1970,
author = {B. Martinet},
journal = {Rev. Fr. d'Inform. Rech. Oper. S\'er. Rouge},
number = {3},
pages = {154--158},
title = {R\'egularisation d'in\'equations variationnelles par approximations successives},
volume = {4},
year = {1970},
}

@article{martinet1972,
author = {B. Martinet},
journal = {C. R. l'Acad.  Sci. S\'er. A},
pages = {163--165},
title = {Determination approch\'ee d'un point fixe d'une application pseudo-contractante},
year = {1972},
volume = {274},
}

@article{rockafellar1976,
author = {R. T. Rockafellar},
title = {Monotone Operators and the Proximal Point Algorithm},
journal = {SIAM J. Control Optim.},
volume = {14},
number = {5},
pages = {877--898},
year = {1976},
}

@article{brezis1978,
year={1978},
issn={0021-2172},
journal={Isr. J. Math.},
volume={29},
number={4},
title={Produits infinis de resolvantes},
author={H. Brezis and P. L. Lions},
pages={329--345},
}

@article{chen2011,
title = {Projection Onto A Simplex},
journal = {arXiv},
author = {Yunmei Chen and Xiaojing Ye},
year = {2011},
}

@article {tibshirani2005,
author = {R. Tibshirani and M. Saunders and S. Rosset and J. Zhu and K. Knight},
title = {Sparsity and smoothness via the fused lasso},
journal = {J. R. Stat. Soc. Ser. B. Stat. Methodol.},
volume = {67},
number = {1},
pages = {91--108},
year = {2005},
}

@article{rapaport2008,
author = {F. Rapaport and E. Barillot and J.-P. Vert},
title = {Classification of {arrayCGH} data using fused {SVM}},
journal = {Bioinformatics},
volume = {24},
number = {13},
pages = {i375--i382},
year = {2008},
}

@article{tibshirani2008,
author = {R. Tibshirani and P. Wang},
title = {Spatial smoothing and hot spot detection for {CGH} data using the fused lasso},
journal = {Biostatistics},
volume = {9},
number = {1},
year = {2008},
pages = {18--29},
}

@article {brodie2009,
	author = {Brodie, Joshua and Daubechies, Ingrid and De Mol, Christine and Giannone, Domenico and Loris, Ignace},
	title = {Sparse and stable {M}arkowitz portfolios},
	volume = {106},
	number = {30},
	pages = {12267--12272},
	year = {2009},
	journal = {Proc. Natl. Acad. Sci. USA}
}

@Article{Combettes2018,
author="Combettes, Patrick L.
and Eckstein, Jonathan",
title="Asynchronous block-iterative primal-dual decomposition methods for monotone inclusions",
journal="Mathematical Programming",
year="2018",
volume="168",
number="1",
pages="645--672",
}

@Article{Eckstein2017,
author="Eckstein, Jonathan",
title="A Simplified Form of Block-Iterative Operator Splitting and an Asynchronous Algorithm Resembling the Multi-Block Alternating Direction Method of Multipliers",
journal="Journal of Optimization Theory and Applications",
year="2017",
volume="173",
number="1",
pages="155--182",
}

@article{raguet2013,
author = {Raguet, H. and Fadili, J. and Peyr\'e, G.},
title = {A Generalized Forward-Backward Splitting},
journal = {SIAM Journal on Imaging Sciences},
volume = {6},
number = {3},
pages = {1199--1226},
year = {2013},
}

@article{johnstone2018,
author = {Patrick R. Johnstone and Jonathan Eckstein},
title = {Projective Splitting with Forward Steps: Asynchronous and Block-Iterative Operator Splitting},
year = {2018},
journal = {arXiv}
}

@article{spingarn1985,
year={1985},
journal={Mathematical Programming},
ajournal={Math. Program.},
volume={32},
number={2},
title={Applications of the method of partial inverses to convex programming: Decomposition},
author={J. E. Spingarn},
pages={199--223},
}

@article{BAFDRS,
author = {Luis M. Brice{\~n}o-Arias},
title = {Forward-{D}ouglas--{R}achford splitting and forward-partial inverse method for solving monotone inclusions},
journal = {Optimization},
volume = {64},
number = {5},
pages = {1239--1261},
year  = {2015},
}

@Article{Raguet2018,
author="Raguet, Hugo",
title="A note on the forward-{D}ouglas--{R}achford splitting for monotone inclusion and convex optimization",
journal="Optimization Letters",
year="2018",
}

@article{combettes2008,
  author={Patrick L. Combettes and Jean-Christophe Pesquet},
  title={A proximal decomposition method for solving convex variational inverse problems},
  journal={Inverse Problems},
  volume={24},
  number={6},
  year={2008},
}

@incollection{combettes2011,
  TITLE = {Proximal Splitting Methods in Signal Processing},
  AUTHOR = {Combettes, Patrick Louis and Pesquet, Jean-Christophe},
  BOOKTITLE = {Fixed-Point Algorithms for Inverse Problems in Science and Engineering},
  EDITOR = {Bauschke, H.H. and Burachik, R.S. and Combettes, P.L. and  Elser, V. and  Luke, D.R. and Wolkowicz, H.},
  PUBLISHER = {Springer},
  PAGES = {185--212},
  YEAR = {2011},
}

@MastersThesis{banert2012,
    author     =     {Sebastian Banert},
    title     =     {A Relaxed Forward-Backward Splitting Algorithm for Inclusions of Sums of Monotone Operators},
    school     =     {Technische Universit\"at Chemnitz},
    year = {2012},
    }

@article{malitsky2018,
title = {A Forward-Backward Splitting Method for Monotone Inclusions Without Cocoercivity},
author = {Yura Malitsky and Matthew K. Tam},
journal = {arXiv},
year = {2018},
}

@article{bot_primal-dual_2013,
author = {Bo{\c{t}}, R. I. and Hendrich, C.},
title = {A {D}ouglas--{R}achford Type Primal-Dual Method for Solving Inclusions with Mixtures of Composite and Parallel-Sum Type Monotone Operators},
journal = {SIAM Journal on Optimization},
volume = {23},
number = {4},
pages = {2541-2565},
year = {2013},
}

@article{bot_portfolio2015,
author = {Radu Ioan Bo{\c{t}} and Christopher Hendrich},
title = {Convex risk minimization via proximal splitting methods},
year = {2015},
volume = {9},
number = {5},
pages = {867--885},
journal = {Optimization Letters},
}

@Article{Combettes2012,
author="Combettes, Patrick L. and Pesquet, Jean-Christophe",
title="Primal-Dual Splitting Algorithm for Solving Inclusions with Mixtures of Composite, {L}ipschitzian, and Parallel-Sum Type Monotone Operators",
journal="Set-Valued and Variational Analysis",
year="2012",
volume="20",
number="2",
pages="307--330",
}

@article{Farkas1902,
author = {Farkas, Julius},
journal = {Journal f\"ur die reine und angewandte Mathematik},
pages = {1--27},
title = {Theorie der einfachen Ungleichungen.},
volume = {124},
year = {1902},
}

@ARTICLE{byrne_poisson_1993, 
author={C. L. Byrne}, 
journal={IEEE Transactions on Image Processing}, 
title={Iterative image reconstruction algorithms based on cross-entropy minimization}, 
year={1993}, 
volume={2}, 
number={1}, 
pages={96--103},
}

@Article{le_poisson_2007,
author="Le, Triet
and Chartrand, Rick
and Asaki, Thomas J.",
title="A Variational Approach to Reconstructing Images Corrupted by {P}oisson Noise",
journal="Journal of Mathematical Imaging and Vision",
year="2007",
volume="27",
number="3",
pages="257--263",
}

@article{chaux_poisson_2009,
author = {Chaux, C. and Pesquet, J. and Pustelnik, N.},
title = {Nested Iterative Algorithms for Convex Constrained Image Recovery Problems},
journal = {SIAM Journal on Imaging Sciences},
volume = {2},
number = {2},
pages = {730--762},
year = {2009},
}

@article{zanella_poisson_2009,
  author={R Zanella and P Boccacci and L Zanni and M Bertero},
  title={Efficient gradient projection methods for edge-preserving removal of {P}oisson noise},
  journal={Inverse Problems},
  volume={25},
  number={4},
  pages={045010},
  year={2009},
}

@article{dinh_farkas_2007,
title={New {F}arkas-type constraint qualifications in convex infinite programming},
volume={13},
number={3},
journal={ESAIM: Control, Optimisation and Calculus of Variations},
author={Dinh, Nguyen and Goberna, Miguel A. and L{\'o}pez, Marco A. and Son, Ta Quang},
year={2007},
pages={580–597}
}

@article{bot_farkas_2005,
author = {Bo{\c{t}}, R. and Wanka, G.},
title = {Farkas-Type Results With Conjugate Functions},
journal = {SIAM Journal on Optimization},
volume = {15},
number = {2},
pages = {540--554},
year = {2005},
}

@ARTICLE{condat2013_denoising, 
author={L. Condat}, 
journal={IEEE Signal Processing Letters}, 
title={A Direct Algorithm for 1-D Total Variation Denoising}, 
year={2013}, 
volume={20}, 
number={11}, 
pages={1054--1057},
}

@inproceedings{Barbero:2011:FNM:3104482.3104522,
 author = {Barbero, {\'A}lvaro and Sra, Suvrit},
 title = {Fast {N}ewton-type Methods for Total Variation Regularization},
 booktitle = {Proceedings of the 28th International Conference on International Conference on Machine Learning (ICML)},
 year = {2011},
 pages = {313--320},
}

@INPROCEEDINGS{6288032, 
author={U. Kamilov and E. Bostan and M. Unser}, 
booktitle={2012 IEEE International Conference on Acoustics, Speech and Signal Processing (ICASSP)}, 
title={Generalized total variation denoising via augmented Lagrangian cycle spinning with Haar wavelets}, 
year={2012}, 
volume={}, 
number={}, 
pages={909--912},
}

@ARTICLE{5981401, 
author={F. I. Karahanoglu and {\.I}. Bayram and D. Van De Ville}, 
journal={IEEE Transactions on Signal Processing}, 
title={A Signal Processing Approach to Generalized 1-D Total Variation}, 
year={2011}, 
volume={59}, 
number={11}, 
pages={5265--5274},
}

@article{WAHLBERG201283,
title = "An {ADMM} Algorithm for a Class of Total Variation Regularized Estimation Problems",
journal = "IFAC Proceedings Volumes",
volume = "45",
number = "16",
pages = "83--88",
year = "2012",
author = "Bo Wahlberg and Stephen Boyd and Mariette Annergren and Yang Wang",
}

@article{johnstone2019,
title = {Convergence Rates for Projective Splitting},
author = {Patrick R. Johnstone and Jonathan Eckstein},
journal = {SIAM Journal on Optimization},
year = {2019},
}

\end{document}